\documentclass[10pt,reqno]{amsart}
\usepackage{amsmath}
\usepackage{amsthm}
\usepackage{amssymb}
\usepackage{graphicx}
\usepackage{amsfonts}
\usepackage{latexsym}
\usepackage{setspace}
\usepackage{mathdots}
\usepackage{mathrsfs}
\usepackage{cite,hyperref}
\usepackage{url}
\usepackage{bm}

\usepackage{tkz-linknodes}

\usepackage[active]{srcltx}

\numberwithin{equation}{section}
\theoremstyle{plain}
\newtheorem{Proposition}[equation]{Proposition}
\newtheorem{Corollary}[equation]{Corollary}
\newtheorem{Theorem}[equation]{Theorem}

\theoremstyle{definition}

\newtheorem{Definition}[equation]{Definition}

\newtheorem{Example}[equation]{Example}
\newtheorem{Remark}[equation]{Remark}

\allowdisplaybreaks

\def\C{\mathbb{C}}
\def\R{\mathbb{R}}
\def\D{\mathbb{D}}
\def\N{\mathbb{N}}
\def\T{\mathbb{T}}

\def\K{\mathcal{K}}
\def\H{\mathcal{H}}

\def\HH{\mathscr{H}}

\def\VV{\mathscr{V}}
\def\SS{\mathscr{S}}

\def\ZZZ{\mathfrak{Z}}

\newcommand{\ran}[1]{\ensuremath{\operatorname{Ran} ( {#1} ) }}
\renewcommand{\ker}[1]{\ensuremath{\operatorname{Ker} ({#1}) }}
\newcommand{\dom}[1]{\ensuremath{\operatorname{Dom} ({#1}) }}
\renewcommand{\dim}[1]{\ensuremath{\operatorname{dim} ( {#1} ) }}

\def\TRng{\widetilde{\operatorname{Ran}}}

\def\Halmos{\precsim}
\def\unitary{\preccurlyeq}
\def\quasi{\preccurlyeq_{q}}

\newcommand{\be}{\begin{equation}}
\newcommand{\ee}{\end{equation}}
\newcommand{\ba}{\begin{eqnarray}}
\newcommand{\ea}{\end{eqnarray}}
\newcommand{\wt}{\ensuremath{\widetilde}}

\renewcommand{\H}{\ensuremath{\mathcal{H} }}

\newcommand{\Z}{\ensuremath{\mathbb{Z} }}

\newcommand{\ip}[2]{\ensuremath{\langle {#1} , {#2} \rangle}}
\renewcommand{\dim}[1]{\ensuremath{\operatorname{dim} \left( {#1} \right) }}

\newcommand{\bn}{\begin{enumerate}}
\newcommand{\en}{\end{enumerate}}

\begin{document}

\bibliographystyle{amsplain}

\title{Partial Orders on Partial Isometries}

 \author{Stephan Ramon Garcia}
   \address{   Department of Mathematics\\
            Pomona College\\
            Claremont, California\\
            91711 \\ USA}
    \email{Stephan.Garcia@pomona.edu}
    \urladdr{\url{http://pages.pomona.edu/~sg064747}}
    \thanks{First author acknowledges support of NSF Grant DMS-1265973. Second author acknowledges support of NRF CPRR Grant 90551.}

\author{Robert T. W. Martin}
\address{Department of Mathematics and Applied Mathematics, University of Cape Town, Cape Town, South Africa}
\email{rtwmartin@gmail.com}

\author[W.T.~Ross]{William T. Ross}
	\address{Department of Mathematics and Computer Science, University of Richmond, Richmond, VA 23173, USA}
	\email{wross@richmond.edu}

\begin{abstract}

    This paper studies three natural pre-orders of increasing generality on the set of all completely non-unitary partial isometries with equal defect indices. We show that the problem of determining when one partial isometry is less than another with respect to these pre-orders is equivalent to the existence of a bounded (or isometric) multiplier between two natural reproducing kernel Hilbert spaces of analytic functions. For large classes of partial isometries these spaces can be realized as the well-known model subspaces and deBranges-Rovnyak spaces. This characterization is applied to investigate properties of these pre-orders and the equivalence classes they generate.

\vspace{5mm}   \noindent {\it Key words and phrases}:
Hardy space, model subspaces, deBranges-Rovnyak spaces, partial isometries, symmetric operators, partial order, pre-order

\vspace{3mm}
\end{abstract}

\maketitle

\section{Introduction}

This paper explores several partial orders on various sets of equivalence classes of partial isometries on Hilbert spaces and their relationship to the function theory problem of when there exists a multiplier from one Hilbert space of analytic functions to another.

More specifically, for $n \in \N \cup \{\infty\}$, we examine the class $\mathscr{V}_{n}(\H)$ of all bounded linear operators $V$ on a complex separable Hilbert space $\H$ satisfying: (i)  $V$ is a partial isometry;
(ii) the defect spaces
$\mathfrak{D}_{+}(V) := \ker{V}$ and $\mathfrak{D}_{-}(V) := \ran{V}^{\perp}$
 have equal dimension $n$;  (iii) there exists no proper reducing subspace $\mathcal{M}$ for $V$  for which $V|_{\mathcal{M}}$ is unitary. An operator satisfying this last condition is said to be {\em completely non-unitary}.
We use the notation $\VV_n$ when considering the set of all $\VV_{n}(\H)$ for any Hilbert space $\H$.

A theorem of Liv\v{s}ic \cite{ MR0113143} settles the unitary equivalence question for $\mathscr{V}_n$. More precisely, to each $V \in \mathscr{V}_n$ there is an associated operator-valued contractive analytic function $w_V$ on $\D$, called the {\em characteristic function} associated with $V$, such that $V_1, V_2 \in \mathscr{V}_n$ are unitarily equivalent if and only if $w_{V_1}$ coincides with $w_{V_2}$. This idea was expanded to contraction operators \cite{MR2760647, MR0244795, MR0215065, MR875237}.

In this paper, we explore three partial orders on $\mathscr{V}_n$ and their possible relationships with  the Liv\v{s}ic characteristic function. After some introductory material, we define three relations $\Halmos, \unitary$, and $\quasi$ on $\mathscr{V}_n$. Each defines a pre-order (reflexive and transitive) and each induces an equivalence relation on $\mathscr{V}_{n}$ by $A \approx B$ if $A \Halmos B$ and $B \Halmos A$ (similarly for the relations $\unitary$ and $\quasi$). In turn, these three equivalence relations generate corresponding equivalence classes $[A]$ of operators in $\mathscr{V}_n$ and induce partial orders on the set of equivalence classes $\mathscr{V}_{n}/\!\approx$.
The first of these partial orders $\Halmos$ was explored by Halmos and McLaughlin \cite{Halmos} and the equivalence classes turn out to be trivial in the sense that $A \Halmos B$ and $B \Halmos A$ if and only if $A = B$. Classifying the equivalence classes induced by $\unitary$ and $\quasi$ is more complicated and requires further discussion.

Our approach to understanding $\unitary$ and $\quasi$ is to recast the problem in terms of the existence of multipliers between spaces of analytic functions. Using ideas from Liv\v{s}ic \cite{ MR0113143} and Kre\u{\i}n \cite{MR0048704} (and explored further by deBranges and Rovnyak in \cite{MR0244795, MR0215065} and by Nikolskii and Vasyunin in \cite{MR875237}), we associate each $V \in \mathscr{V}_n$ with a
Hilbert space $\mathscr{H}_{V}$ of vector-valued analytic functions on $\C \setminus \mathbb{T}$ such that $V|_{\ker{V}^{\perp}}$ is unitarily equivalent to $\ZZZ_{V}$, where $\ZZZ_{V} f = z f$ on
$\dom{\ZZZ_{V}} = \{f \in \mathscr{H}_{V}: z f \in \mathscr{H}_{V}\}$. We show, for $V_1, V_2 \in \mathscr{V}_{n}$, that (i) $V_{1}$ is unitarily equivalent to $V_2$ if and only if there is an isometric multiplier from $\mathscr{H}_{V_1}$ onto $\mathscr{H}_{V_2}$ (more precisely, there exists an operator-valued analytic function $\Phi$ on $\C \setminus \T$ such that $\Phi \mathscr{H}_{V_1} = \mathscr{H}_{V_2}$ and the operator $f \mapsto \Phi f$ from $\mathscr{H}_{V_1}$ to $\mathscr{H}_{V_2}$ is isometric); (ii) $V_{1} \unitary V_2$ if and only if there is an isometric multiplier $\Phi$ from $\mathscr{H}_{V_1}$ into $\mathscr{H}_{V_2}$; (iii) $V_{1} \quasi V_2$ if and only if there is a multiplier $\Phi$ from $\mathscr{H}_{V_1}$ into $\mathscr{H}_{V_2}$ (that is, $\Phi \mathscr{H}_{V_1} \subset \mathscr{H}_{V_2}$).

What makes this partial order problem interesting from a complex analysis perspective is that under certain circumstances, depending on the Liv\v{s}ic function, the partial order problem (When is $A \unitary B$? When is $A \quasi B$?) can be also rephrased in terms of the existence of (isometric) multipliers from one model space $(\Theta H^2)^{\perp}$ to another, or perhaps from one de Branges-Rovnyak $\mathscr{H}(b)$ space to another. These are well-known and well-studied Hilbert spaces of analytic functions on $\D$ which have many connections to operator theory \cite{MR2760647, MR0244795, MR0215065, MR875237}.

\section{Partial isometries}

Let $\mathcal{B}(\H)$ denote the set of all bounded operators on a separable complex Hilbert space $\H$.
\begin{Definition}
A operator $V \in \mathcal{B}(\H)$ is called a {\em partial isometry} if $V|_{\ker{V}^{\perp}}$ is an isometry.
The space $\ker{V}^{\perp}$ is called the {\em initial space} of $V$ while $\ran{V}$ is called the {\em final space} of $V$. The spaces
$\mathfrak{D}_{+}(V) := \ker{V}$ and  $\mathfrak{D}_{-}(V) := \ran{V}^{\perp}$ are called the {\em defect spaces} of $V$ and the pair of numbers $(n_+, n_- )$, where $n_{+}$ and $n_{-}$ are the corresponding dimensions of $\mathfrak{D}_{+}(V)$ and $\mathfrak{D}_{-}(V)$, are called the {\em deficiency indices} of $V$.
\end{Definition}

Note that a partial isometry $V$ with deficiency indices $(0,0)$ is a unitary operator. The following proposition is standard and routine to verify.

\begin{Proposition}\label{basicPIfacts}
For $V \in \mathcal{B}(\H)$ the following are equivalent:
\begin{enumerate}
\item $V$ is a partial isometry;
\item $V = V V^{*} V$;
\item $V^{*}$ is a partial isometry;
\item $V^{*} V$ is an orthogonal projection;
\item $V V^{*}$ is an orthogonal projection.
\end{enumerate}
\end{Proposition}

One can show that $V^{*} V$ is the orthogonal projection of $\H$ onto the initial space of $V$ while $V V^{*}$ is the orthogonal projection of $\H$ onto the final space of $V$. Note that if $V$ is a partial isometry, then $Q_1 V Q_2$ is also a partial isometry for any unitary operators $Q_1, Q_2$ on $\H$.

When $\dim{\H} < \infty$, the partial isometries $V$ on $\H$ are better understood \cite{MR0207723, MR0227194, MR0225790}. Here we think of $V \in M_{n}(\C)$.  For example, if $\{\mathbf{u}_{1}, \ldots, \mathbf{u}_{n}\}$ is an orthonormal basis for $\C^n$ then for any $1 \leqslant r \leqslant n$ the (column partitioned) matrix
\begin{equation}\label{weoiruwoeiru}
[\mathbf{u}_{1}|\mathbf{u}_{2}|\cdots|\mathbf{u}_{r}|\mathbf{0}|\mathbf{0}|\cdots|\mathbf{0}]
\end{equation}
 is a partial isometry with initial space $\bigvee\{\mathbf{e}_1, \ldots, \mathbf{e}_{r}\}$ (where $\mathbf{e}_j$ is the $j$th standard basis vectors for $\C^n$ and $\bigvee$ is the linear span) and final space $\bigvee\{\mathbf{u}_1, \ldots, \mathbf{u}_{r}\}$. For any $n \times n$ unitary matrix $Q$
\begin{equation}\label{QUQ}
Q [\mathbf{u}_{1}|\mathbf{u}_{2}|\cdots|\mathbf{u}_{r}|\mathbf{0}|\mathbf{0}|\cdots|\mathbf{0}] Q^{*}
\end{equation} is also a partial isometry.

\begin{Proposition}\label{canonicalformPI}
For $V \in M_{n}(\C)$, the following are equivalent:
\begin{enumerate}
\item $V$ is a partial isometric matrix.
\item $V = Q [\mathbf{u}_{1}|\mathbf{u}_{2}|\cdots|\mathbf{u}_{r}|\mathbf{0}|\mathbf{0}|\cdots|\mathbf{0}] Q^{*}$,
where $\{\mathbf{u}_1, \ldots, \mathbf{u}_r: 1 \leqslant r \leqslant n\}$ is a set of orthonormal vectors in $\C^n$ and $Q$ is a unitary matrix.
\item $V = U P$, where $U$ is a unitary matrix and $P$ is an orthogonal projection.
\end{enumerate}
\end{Proposition}

The unitary matrix $U$ in the proposition above is not unique and is often called a {\em unitary extension} of $V$.  For general partial isometries $V$ on possibly infinite dimensional Hilbert spaces $\H$, unitary extensions in $\mathcal{B} (\H )$ need not always exist. However,  we know exactly when this happens \cite{MR1255973}.

\begin{Proposition}\label{has-ue}
A partial isometry $V \in \mathcal{B}(\H)$ has unitary extensions in $\mathcal{B} (\H)$ if and only if $V$ has equal deficiency indices.
\end{Proposition}

When $\dim{\H} < \infty$, deficiency indices are always equal.

\begin{Definition}
A partial isometry $V \in \mathcal{B}(\mathcal{H})$ is {\em completely non-unitary} if there is no nontrivial reducing subspace $\mathcal{M}$ of $\mathcal{H}$ (that is, $V \mathcal{M} \subset \mathcal{M}$ and $V^{*} \mathcal{M} \subset \mathcal{M}$) such that $V|_{\mathcal{M}}$ is a unitary operator on $\mathcal{M}$.
\end{Definition}

It is well-known \cite{MR2760647} that every partial isometry $V$ can be written as $V = V_1 \oplus V_2$, where $V_1$ is unitary and $V_2$ is completely non-unitary. In finite dimensions it is easy to identify the completely non-unitary partial isometries.

\begin{Proposition}\label{CNU-Mn}
A partially isometric matrix $V \in M_{n}(\C)$ is completely non-unitary if and only if all of its eigenvalues lie in the open unit disk $\D$.
\end{Proposition}





\begin{Remark}\label{Livsic-domains}
Our launching point here is the work of Liv\v{s}ic \cite{MR0113143} which explores this same material in a slightly different way. Liv\v{s}ic considers isometric operators $\widehat{V}$ that are defined on a domain $\dom{\widehat{V}}$ on a Hilbert space $\H$ such that $\widehat{V}$ is isometric on $\dom{\widehat{V}}$. Here, the defect spaces are defined to be $\dom{\widehat{V}}^{\perp}$ and $(\widehat{V} \dom{\widehat{V}})^{\perp}$. If we define
\begin{equation}\label{LivPI}
V \mathbf{x} =
\begin{cases} \widehat{V}  \mathbf{x} &\mbox{if } \mathbf{x} \in \dom{\widehat{V}}, \\
0 & \mbox{if } \mathbf{x} \in \dom{\widehat{V}}^{\perp}, \end{cases}
\end{equation}
then $V$ is a partial isometry with initial space $\dom{\widehat{V} }^{\perp}$ and final space $\widehat{V} \dom{\widehat{V}}$. Conversely, if $V$ is a partial isometry, then $\widehat{V} = V|_{\ker{V}^{\perp}}$ is an isometric operator in the Liv\v{s}ic setting.
\end{Remark}

\begin{Remark}
The discussion of unitary equivalence and partial orders in this paper focuses on partial isometries. However, using some standard theory,  all of our results have analogues expressed in terms of unbounded symmetric linear transformations \cite{MR1255973}. Indeed, let
$$\beta(z) = \frac{z - i}{z + i}$$ denote
the {\em Cayley transform}, a fractional linear transformation that maps the upper half plane $\C_{+}$ bijectively to $\D$ and $\R$ bijectively onto $\T \setminus \{1 \}$. Here $\T$ denotes the unit circle in $\C$. Notice that
$$\beta^{-1}(z) = i \frac{1 + z}{1 - z}.$$
If $V$ is a partial isometry, the operator
$$S := \beta^{-1}(V) = i (I + V) (I - V)^{-1}$$ is an unbounded, closed, symmetric linear transformation with domain $$\dom{S} = (I - V) \ker{V}^{\perp}.$$ Note that if $1$ is an eigenvalue of $V$,
then $S = \beta^{-1} (V)$ is not densely defined.  This poses no major technical difficulties in our analysis, but it is something to keep in mind. See \cite{Habock,ST10} for references on symmetric linear transformations which are not necessarily densely defined. We will reserve the term {\em symmetric operator} for a densely defined
symmetric linear transformation.

Conversely, if $S$ is a symmetric linear transformation with domain $\dom{S}$, then
$$\beta(S) = (S - i I)(S + i I)^{-1}$$ is an isometric operator on the domain $(S + i I) \dom{S}$ that can be extended to a partial isometry on all of $\H$ by extending it to be zero on the orthogonal complement of its domain.
A closed symmetric linear transformation is said to be \emph{simple} if its Cayley transform $V = \beta(S)$ is completely non-unitary.  This happens if and only if $S$ has no self-adjoint restriction to the intersection of its domain with a proper, nontrivial invariant subspace.

Note that $V$ has unitary extensions if and only if $\beta^{-1}(V)$ has self-adjoint extensions. The Cayley transform shows that if $V = \beta(S)$, the deficiency subspaces $\ker{V}$ and $\ran{V}^{\perp}$ are equal to the deficiency spaces $\ran{S - i I}^{\perp}$ and $\ran{S + i I}^{\perp}$, respectively.
\end{Remark}

Let us give some examples of partial isometries that will be useful later on.

\begin{Example}\label{12384u234}
\begin{enumerate}
\item The matrices $Q [\mathbf{u}_{1}|\mathbf{u}_{2}|\cdots|\mathbf{u}_{r}|\mathbf{0}|\mathbf{0}|\cdots|\mathbf{0}] Q^{*}$ from \eqref{QUQ} are all of the partial isometries on $\C^n$.
\item Every orthogonal projection is a partial isometry. However, no orthogonal projection is completely non-unitary.
\item The {\em unilateral shift}
$S :H^2 \to H^2$, $S f = z f$, on the Hardy space $H^2$ \cite{Duren} is a partial isometry with initial space $H^2$ and final space
$H^{2}_{0} := \{f \in H^2: f(0) = 0\}$. The defect spaces are
$\mathfrak{D}_{+}(S) = \{0\}$, $\mathfrak{D}_{-}(S) = \C$ and thus the deficiency indices of $S$ are $(0, 1)$. Since the indices are not equal, $S$ does not have unitary extensions to $H^2$ (Proposition \ref{has-ue}).
\item The adjoint $S^{*}$ of $S$ is given by
$S^{*} f = \frac{f - f(0)}{z}$ and it is called the {\em backward shift}. Note that $S^{*}$ is a partial isometry (Proposition \ref{basicPIfacts}) with initial space $H^{2}_{0}$ and final space $H^2$. The defect spaces are
$\mathfrak{D}_{+}(S^{*}) = \C$ and $\mathfrak{D}_{-}(S^{*}) = \{0\}$ and thus the deficiency indices are $(1, 0)$.
Thus the backward shift $S^{*}$ has no unitary extensions to $H^2$.
\item The operator
$S^{*} \oplus S: H^2 \oplus H^2 \to H^2 \oplus H^2$ is a partial isometry with initial space $H^{2}_{0} \oplus H^2$ and final space $H^2 \oplus H^{2}_{0}$. The defect spaces are
$\mathfrak{D}_{+}(S^{*} \oplus S) = \C \oplus \{0\}$ and $\mathfrak{D}_{-}(S^{*} \oplus S) = \{0\} \oplus \C$ and thus $S^{*} \oplus S$ has deficiency indices $(1, 1)$. One can show that this operator is also completely non-unitary and thus $S^{*} \oplus S \in \mathscr{V}_{1}(H^2)$.
\item Consider the partial isometry $S \otimes S^*$ acting on $\H := H^2 \otimes H^2$. Alternatively, this operator can be viewed as the operator block matrix
\be  \begin{bmatrix} 0 & &  &  &  \\  S^* & 0 &  &  & \\ & S^* & 0 &  &  \\ & & \ddots & \ddots &  \\ & & & & \end{bmatrix}  \label{blockop} \ee acting on the Hilbert space
$$ \H := \bigoplus _{k \geqslant 0} H^2. $$
One can verify that
$$\ker{S\otimes S^* } ^\perp = \H _0 := \bigoplus _{k \geqslant 0} H^2 _0,$$
$$\ran{S \otimes S ^* } ^\perp = \ker{S^* \otimes S } =  H ^2 \oplus \bigoplus _{k \geqslant 1}  \{ 0 \}, $$
and that $S \otimes S^{*}$ is completely non-unitary. Thus $S \otimes S^* \in \VV _\infty (H^2 \otimes H^2 )$.
\item If $\Theta$ is an inner function, define $\K_{\Theta} = (\Theta H^2)^{\perp}$ to be the well-known model space.  Consider the compression $S_{\Theta} := P_{\Theta} S|_{\K_{\Theta}}$, of the shift to $\K_{\Theta}$, where $P_{\Theta}$ is the orthogonal projection of $L^2$ onto $\K_{\Theta}$.
If $\Theta(0) = 0$, one can show that
$$\mathfrak{D}_{+}(S_{\Theta}) = \ker{S_{\Theta}} = \C \frac{\Theta}{z}, \quad \mathfrak{D}_{-}(S_{\Theta}) = (\ran{S_{\Theta}} )^{\perp} = \C.$$
Furthermore,
$\ker{S_{\Theta}}^{\perp} = \{f \in \K_{\Theta}: z f \in K_{\Theta}\}$ and so
$S_{\Theta}$ is isometric on $\ker{S_{\Theta}}^{\perp}$. Thus $S_{\Theta}$ is a partial isometry with defect indices $(1, 1)$. It is well-known that the compressed shift $S_{\Theta}$ is irreducible (has no nontrivial reducing subspaces) and thus $S_{\Theta}$ is completely non-unitary. Hence, assuming $\Theta(0) = 0$, $S_{\Theta} \in \mathscr{V}_{1}(\K_{\Theta})$.
The model space $\K_{\Theta}$ is a reproducing kernel Hilbert space with kernel
$$k^{\Theta}_{\lambda} = \frac{1 - \overline{\Theta(\lambda)}\Theta}{1 - \overline{\lambda} z}.$$
To every model space there is a natural conjugation $C_{\Theta}$ defined via the radial (or non-tangential) boundary values of $f$ and $\Theta$ by
$C_{\Theta} f  = \Theta \overline{\zeta f}$. One can see that $C_{\Theta}$ is conjugate linear, isometric, and involutive.  Furthermore, a calculation shows that
$$C_{\Theta} k^{\Theta}_{\lambda}  = \frac{\Theta - \Theta(\lambda)}{z - \lambda}.$$
The compressed shift $S_{\Theta}$ also obeys the property
$S_{\Theta} = C_{\Theta} S_{\Theta}^* C_{\Theta}$.
This puts $S_{\Theta}$ into a class of operators called {\em complex symmetric operators} \cite{CSO,CSO2, MPCSO}.
Furthermore, $S_{\Theta}^{*} = S^{*}|_{\K_{\Theta}},$
the restriction of the backward shift to the model space $\K_{\Theta}$.
\item Another partial isometry on $\K_{\Theta}$ closely related to $S_{\Theta}$ is created as follows.
The operator $\widehat{M}_{\Theta} f = z f$ is not a well defined operator on all of $\K_{\Theta}$, but it is defined on
$\dom{\widehat{M}_{\Theta}} = \{f \in \K_{\Theta}: z f \in \K_{\Theta}\}$.
A little thought shows that
$\dom{\widehat{M} _{\Theta}} = \{f \in \K_{\Theta}: (C_{\Theta} f)(0) = 0\}$.
Using the isometric nature of $C_{\Theta}$ and the fact that point evaluations are continuous, we see that $\dom{\widehat{M}_{\Theta}}$ is closed. Furthermore, we know that
$\ran{\widehat{M} _{\Theta}} = \widehat{M_{\Theta}} \dom{\widehat{M_{\Theta}}} = \{f \in \K_{\Theta}: f(0) = 0\}$. Keeping with our previous notation from Remark \ref{Livsic-domains}, let $M_{\Theta}$ be the operator that is equal to $\widehat{M}_{\Theta}$ on $\dom{\widehat{M}_{\Theta}}$ and equal to zero on $\dom{\widehat{M}_{\Theta}}^{\perp}$. Observe that $M_{\Theta}$ is a partial isometry with initial space $\dom{\widehat{M} _{\Theta}}$ and final space $\ran{\widehat{M} _{\Theta}}$. Furthermore, the defect spaces are
$$\mathfrak{D}_{+}(M_{\Theta}) = \C C_{\Theta} k^{\Theta}_{0}, \quad \mathfrak{D}_{-}(M_{\Theta}) = \C k^{\Theta}_{0}$$
so that $M_{\Theta} \in \mathscr{V}_{1}(\K_{\Theta})$. In fact, if $\Theta(0) = 0$, then
$M_{\Theta} = S_{\Theta}$. We will see in Example \ref{FShift} below that for any $a \in \D$,
$M_{\Theta} \cong M_{\Theta_a}$, where
$$\Theta_{a} := \frac{\Theta - a}{1 - \overline{a} \Theta}.$$
Thus for any inner $\Theta$, $M_{\Theta} \cong S_{\Theta_{\Theta(0)}}$.
\item For $b \in H^{\infty}_{1} := \{g \in H^{\infty}: \|g\|_{\infty} \leqslant 1\}$, the closed unit ball in $H^{\infty}$, define
$\HH(b)$, the {\em deBranges-Rovnyak} space to be the reproducing kernel space corresponding to the kernel
$$k^{b}_{\lambda}  = \frac{1 - \overline{b(\lambda)} b}{1 - \overline{\lambda} z}, \quad \lambda, z \in \D.$$ When $\|b\|_{\infty} < 1$, $\HH(b) = H^2$ with an equivalent norm. On the other extreme, when $b$ is an inner function,  $\HH(b)$ is the model space $\K_b$ with the standard $H^2$ norm \cite{Sarason}.

The analogue of the compressed shift $S_{\Theta}$ can be generalized to the case where $b$ is an extreme point of the unit ball of $H^\infty$, but not to the case where $b$ is not an extreme point. To see this, note from \cite[II-7]{Sarason} that $S^{*} \HH(b) \subset \HH(b)$.
If $X  = S^{*}|_{\HH(b)},$
then it was shown in \cite[II-9]{Sarason} that
$X^{*} f = S f - \langle f, S^{*} b\rangle_{b} b.$
If we define
$\HH_{0}(b) = \{f \in \HH: f(0) = 0\}$, we can use the formula above for $X^{*}$ to get
\begin{align*}
X^{*} X f  = X^{*} S^{*} f
 = S S^{*} f - \langle S^{*} f, S^{*} b\rangle_{b} b= f - \langle S^{*} f, S^{*} b\rangle_b b.
\end{align*}
 Since $b$ is an extreme point, $b \notin \HH (b)$ by \cite[V-3]{Sarason}, and it follows that
 $\langle S^{*} f, S^{*} b\rangle_b b = 0$. Thus
 $X^{*}|_{S^{*} \HH_{0}(b)} = S|_{S^{*} \HH_{0}(b)}$ and
 $\widehat{M} _{b} := S|_{S^{*} \HH_{0}(b)}$ is an isometry from $S^{*} \HH_{0}(b)$ onto $\HH_{0}(b)$. A little thought shows that
 $\{f \in \HH(b): z f \in \HH(b)\} = S^{*} \HH_{0}(b)$ and so $\widehat{M} _b$ is multiplication by the independent variable on
 $$\dom{\widehat{M} _b}  = \{f \in \HH(b): z f \in \HH(b)\}.$$

One also has
 $$\ran{\widehat{M} _b}^{\perp} = \C k^{b}_{0}.$$
 Furthermore, using the fact that $\langle S^{*} f, S^{*} b\rangle_{b} = 0$ for all $f \in S^{*} \HH_{0}(b) = \dom{\widehat{M}_b}$, we see that
 $$\dom{\widehat{M} _b}^{\perp} = \C S^{*} b.$$
 This means that the extension operator $M_{b}$ from Remark \ref{Livsic-domains} is a partial isometry with $(1, 1)$ deficiency indices.
One can also show that $M_{b}$ is completely non-unitary and thus $M_{b} \in \VV _1 (\HH (b) )$ whenever $b$ is extreme. The analysis above breaks down when $b$ is non-extreme.

\end{enumerate}
\end{Example}

\section{Abstract model spaces}
\label{absmodel}

In this section we put some results from \cite{MR0113143, MR0048704, MR0244795, MR0215065, MR875237} in a somewhat different context and show  that for $V \in \VV_n(\H)$, $n < \infty$, and {\em model} $\Gamma$ for $V$ (which we will define momentarily) there is an associated reproducing kernel Hilbert space of $\C^n$-valued analytic functions $\HH_{V, \Gamma}$ on $\C \setminus \T$ such that
$$V|_{\ker{V}^{\perp}} \cong \widehat{\ZZZ} _{V, \Gamma},$$ where
$$\widehat{\ZZZ}_{V, \Gamma}: \dom{\widehat{\ZZZ} _{V,\Gamma}} \to \HH_{V}, \quad \widehat{\ZZZ}_{V, \Gamma} = z f,$$
$$\dom{\widehat{\ZZZ}_{V, \Gamma}} = \{f \in \HH_{V, \Gamma}: z f \in \HH_{V, \Gamma}\}.$$ As before (Remark \ref{Livsic-domains}), $\ZZZ _{V, \Gamma}$ is the partial isometric extension of $\widehat{\ZZZ} _{V, \Gamma}$.  This idea was used in a recent paper \cite{AMR}, in the setting of symmetric operators, but we outline the idea here.

Since $V \in \VV_n(\H)$, it has equal deficiency indices and we know from Proposition \ref{has-ue} that $V$ has a unitary extension $U$. For $V \in \VV_{n}(\H)$ Liv\v{s}ic \cite{MR0113143} defines, for each $z \in \C \setminus \T$, the isometric linear transformation $V_z$ by
$$V_{z} f:= (V - z I) (I - \overline{z} V)^{-1} f, \quad f \in (I - \overline{z} V)|\ker{V}^{\perp}.$$ Extend this definition to all of $\H$ by making $V_z$ equal to zero on
$((I - \overline{z} V)|\ker{V}^{\perp})^{\perp}$.
This will define a partial isometry whose initial space is
$$\ker{V_z}^{\perp}= (I - \overline{z} V)\ker{V}^{\perp}.$$
We we define
$$\widetilde{\operatorname{Ran}}(V - z I) := (V - z I) \ker{V}^{\perp},$$
 the final space of $V_{z}$ is
$$\TRng(V_z) := V_z \ker{V_z} ^\perp = \ran{V_z}.$$

\begin{Proposition}
For each $z \in \C \setminus \T$, we have $\widetilde{\operatorname{Ran}}(V_z) = \widetilde{\operatorname{Ran}}(V - z I)$.
\end{Proposition}

\begin{proof}
Note that $\ker{V_z}^{\perp} = (I - \overline{z} V)\ker{V}^{\perp}$ and so
\begin{align*}
 &\TRng(V_z) = V_{z} \ker{V_z}^{\perp}\\
& = (V - z I)(I - \overline{z} V)^{-1} (I - \overline{z} V)\ker{V}^{\perp}\\
& = (V - z I)\ker{V}^{\perp} = \TRng(V - z I) \qedhere
\end{align*}
\end{proof}

\begin{Proposition}[Liv\v{s}ic]
For each $z \in \C \setminus \T$ and unitary extension $U$ of $V$ we have
$$(I - \overline{z} U)^{-1} \TRng (V) ^{\perp} = \TRng (V_z) ^{\perp}.$$
\end{Proposition}

\begin{proof}
	Suppose $f \in \TRng(V) ^\perp$. By the previous proposition, $\TRng(V_z) =\TRng(V-zI) = (V-zI) \ker{V} ^\perp$. Hence,
\begin{align*}
\ip{(I-\overline{z} U) ^{-1} f }{\TRng(V_z)} & =  \ip{U^* (U^* -\overline{z}) ^{-1} f }{(V-zI) \ker{V} ^\perp} \nonumber \\
& =  \ip{f}{ U\ker{V} ^\perp} = \ip{f}{\TRng(V)} \nonumber  =  0. \nonumber \qedhere
\end{align*}
\end{proof}

We now follow a construction in \cite{AMR}. The proofs there are in the setting of symmetric operators but the proofs but carry over to our setting.  Indeed, since $\TRng{V}^{\perp}$ is an $n$-dimensional vector space, we let
$$j: \C^n \to \TRng{V}^{\perp}$$ be any isomorphism and define
$$\Gamma(\lambda) := (I - \overline{\lambda} U)^{-1} \circ j.$$ Then
$\Gamma: \C \setminus \T \to \mathcal{B}(\C^n, \H)$ is anti-analytic and, for each $\lambda \in \C \setminus \T$,
$$\Gamma(\lambda): \C^n \to \TRng(V - \lambda I)^{\perp}$$ is invertible. We also see that $\Gamma(z)^{*} \Gamma(\lambda)$ is invertible for $z, \lambda \in \D$ or $z, \lambda \in \D_{e}$, where
$\D_{e} := \C \setminus \D^{-}$. Finally, as discussed in \cite{AMR}, the fact that $V$ is completely non-unitary (so that $\beta ^{-1} (V)$ is simple) implies $$\bigvee_{\lambda \in \C \setminus \T} \ran{\Gamma(\lambda)} = \H.$$ For any $f \in \H$ define
$$\widehat{f}(\lambda) = \Gamma(\lambda)^{*} f, $$ and let
$$\HH_{V, \Gamma}:= \{\widehat{f}: f \in \H\}.$$
When endowed with the inner product
$$\langle \widehat{f}, \widehat{g} \rangle_{\HH_{V, \Gamma}} := \langle f, g \rangle,$$
 $\HH_{V, \Gamma}$ becomes a $\C^n$-valued Hilbert space of analytic functions on $\C \setminus \T$ such that the operator
$$f \mapsto \widehat{f}$$ is a unitary operator from $\H$ onto $\HH_{V, \Gamma}$ which induces the unitary equivalence $$V|_{\ker{V}^{\perp}} \cong  \widehat{\ZZZ} _{V, \Gamma},$$ where the isometric linear transformation $\widehat{\ZZZ} _{V, \Gamma}$
acts as multiplication by the independent variable on $\HH_V$ on the domain
$$\dom{\widehat{\ZZZ} _{V, \Gamma}} = \{f \in \HH_{V, \Gamma}: z f \in \HH_{V, \Gamma}\}.$$  Note that the hypothesis that $V$ is completely non-unitary is needed here for the inner product on $\HH_{V, \Gamma}$ to be meaningfully defined,
see \cite{AMR} for details.

\begin{Definition}
$\HH_{V, \Gamma}$ is the {\em abstract model space} for $V$ induced by the model $\Gamma$.
\end{Definition}

\begin{Remark} We are not constrained by the above Liv\v{s}ic trick in selecting our model $\Gamma$ for $V$. There are other methods of constructing a model \cite{AMR}. For example, we can use Grauert's Theorem, as was used to prove a related result for bounded operators in a paper of Cowen and Douglas \cite{MR501368}, to find a anti-analytic vector-valued function
$$\gamma(\lambda) := (\gamma(\lambda)_1, \cdots, \gamma(\lambda)_{n}), \quad \lambda \in \C \setminus \T,$$
where $\{\gamma(\lambda)_1, \cdots, \gamma(\lambda)_{n}\}$ is a basis for $\TRng(V - \lambda I)^{\perp}$. Then, if $\{\mathbf{e}_j\}_{j = 1}^{n}$ is the standard basis for $\C^n$,  we can define our abstract model for $V$ to be
\begin{equation}\label{G-trick}
\Gamma(\lambda) := \sum_{j = 1}^{n} \gamma(\lambda)_{j} \otimes \mathbf{e}_j.
\end{equation}
 \end{Remark}

An abstract model space for $\HH_{V, \Gamma}$ is not unique. However, if $\HH_{V, \Gamma}$ and $\HH_{V, \Gamma'}$ are two abstract model spaces for $V$ determined by the models $\Gamma$ and $\Gamma'$, then
$\HH_{V, \Gamma'}= \Theta \HH_{V, \Gamma}$
for some analytic matrix-valued function on $\C \setminus \T$.
Via this multiplier $\Theta$ one can often realize, by choosing the model in a particular way,  $\HH_{V, \Gamma}$ as a certain well-known space of analytic functions such as a model space, de Branges-Rovnyak space, or a Herglotz space. We will get to this in a moment. For now we want to keep our discussion as broad as possible.

\begin{Remark}
Since $\widehat{\ZZZ} _{V, \Gamma} \cong V|_{\ker{V}^{\perp}}$,  $\widehat{\ZZZ}_{V, \Gamma}$ is isometric on $\dom{\widehat{\ZZZ} _{V, \Gamma}}$. As discussed in Remark \ref{Livsic-domains}, we need to think of $\widehat{\ZZZ}_{V, \Gamma}$ as a partial isometry on $\HH_{V, \Gamma}$. We can do this by extending $\widehat{\ZZZ} _{V, \Gamma}$ to all of $\HH_{V, \Gamma}$ so that the
extended operator $\ZZZ_{V, \Gamma}$ on $\HH_{V, \Gamma}$ is a partial isometry with
$$\ker{\ZZZ_{V, \Gamma}}^{\perp} = \dom{\widehat{\ZZZ} _{V, \Gamma}}.$$ The unitary equivalence of $V|_{\ker{V}^{\perp}}$ and $\widehat{\ZZZ} _{V, \Gamma}$ can be extended to a unitary equivalence of $V$ and $\ZZZ_{V, \Gamma}$.
\end{Remark}

The representing space
$$\HH_{V, \Gamma} =: \HH$$
 turns out to be a reproducing kernel Hilbert space with  reproducing kernel
$$k^{\mathscr{H}}_{w}(z) = \Gamma(z)^{*} \Gamma(w), \quad w, z \in \C \setminus \T.$$
This kernel is $M_{n}(\C)$-valued for each $w, z \in \C \setminus \T$, is analytic in $z$, and anti-analytic in $w$. By the term reproducing kernel we mean that for any ($\C^n$-valued) $f \in \mathscr{H}$ and any $w \in \C \setminus \T$ we have
$$\langle f, k^{\mathscr{H}}_{w}(\cdot) \mathbf{a} \rangle_{\mathscr{H}} = \langle f(w), \mathbf{a}\rangle_{\C^{n}} \quad \forall \mathbf{a} \in \C^n.$$ In the above $\langle \cdot, \cdot \rangle_{\mathscr{H}}$ is the inner product in the Hilbert space $\mathscr{H}$ while $\langle \cdot, \cdot \rangle_{\C^n}$ is the standard inner product on $\C^n$.

Also note that the space $\mathscr{H}$ has the division property in that if $f \in \mathscr{H}$ and $f(w) = \mathbf{0}$, then $(z - w)^{-1} f \in \mathscr{H}$.  This means that for any $w \in \C \setminus \T$, there is a $f \in \mathscr{H}$ for which $f(w) \not = \mathbf{0}$. From the reproducing kernel identity above, we see that for any $w \in \C \setminus \T$, the span of
$\{k^{\mathscr{H}}_{w}(\cdot) \mathbf{a}: \mathbf{a} \in \C^n\}$ is an $n$-dimensional subspace of $\mathscr{H}$. Such a kernel is said to be non-degenerate.

Let us give a few examples of abstract models for some of the partial isometries from Example \ref{12384u234}. There is a more canonical choice of model space for $V$. We will see this in the next section.

\begin{Example}[Classical Model spaces]
\label{scalarmodel}
Recall the model space $\K_\Theta = (\Theta H^2)^{\perp}$ and the operator $\widehat{M} _{\Theta}$. If we understand that
$$\ran{\widehat{M} _{\Theta} - \lambda I} = (\widehat{M} _{\Theta} - \lambda I) \dom{\widehat{M} _{\Theta}},$$ then
$$\TRng(M_{\Theta} - \lambda I) = \ran{\widehat{M} _{\Theta} - \lambda I}.$$

Standard computations show that
$$\ran{\widehat{M} _{\Theta} - \lambda I}^{\perp} = \C k^{\Theta}_{\lambda}, \qquad |\lambda| < 1,$$
$$\ran{\widehat{M} _{\Theta} - \lambda I}^{\perp} = \C C_{\Theta} k^{\Theta}_{1/\overline{\lambda}}, \qquad |\lambda| > 1,$$
where $C_{\Theta}$ is the conjugation on $\K_{\Theta}$ discussed in Example \ref{12384u234}.

Using our trick from \eqref{G-trick}, we can compute the abstract model for the partial isometry $M_{\Theta}$ by defining
$$\Gamma(\lambda) = \begin{cases} k^{\Theta}_{\lambda} \otimes 1 &\mbox{if } |\lambda| < 1, \\
C_{\Theta} k^{\Theta}_{1/\overline{\lambda}} \otimes 1 & \mbox{if } |\lambda| > 1.
 \end{cases}$$
 Our abstract model space for $M_{\Theta}$ corresponding to $\Gamma$ is thus
 $$\HH_{M_{\Theta}, \Gamma} = \{\Gamma^{*} f: f \in \K_{\Theta}\}.$$
Observe that
 $$\Gamma(\lambda)^{*} f =  \begin{cases} \langle f, k^{\Theta}_{\lambda}\rangle  &\mbox{if } |\lambda| < 1, \\
\langle f, C_{\Theta} k^{\Theta}_{\lambda} \rangle  & \mbox{if } |\lambda| > 1,
\end{cases}$$ which becomes
$$\Gamma(\lambda)^{*} f = \begin{cases} f(\lambda) &\mbox{if } |\lambda| < 1, \\
\overline{(C_{\Theta} f)(1/\overline{\lambda})} & \mbox{if } |\lambda| > 1.
 \end{cases}$$
\end{Example}

\begin{Example}[Vector-valued model spaces]
For a $n \times n$ matrix-valued inner function $\Theta$ on $\D$ we can define the vector-valued model space by
$$\K_{\Theta} = (\Theta H^{2}_{\C^n})^{\perp},$$
where $H^{2}_{\C^n}$ is the vector-valued Hardy space of $\D$.
The reproducing kernel for $\K_{\Theta}$ is the matrix
$$K^{\Theta}_{\lambda}(z) = \frac{1 - \Theta(\lambda)^{*} \Theta(z)}{1 - \overline{\lambda} z},$$
meaning
$$\langle f(\lambda), \mathbf{a} \rangle_{\C^n} = \langle f, K^{\Theta}_{\lambda} \mathbf{a}\rangle_{H^{2}_{\C^n}}, \quad \lambda \in \D, \mathbf{a} \in \C^n, f \in H^{2}_{\C^n}.$$
There is a conjugation of sorts here, defined by
$$C_{\Theta}: \K_{\Theta} \to K_{\Theta^{T}}, \quad (C_{\Theta} f)(\zeta) = \Theta^{T}(\zeta) (\dagger f)(\zeta),$$
where $\dagger$ is component-wise Schwarz reflection and $T$ denotes the transpose.
As before, one can show that
$$\ran{\widehat{M} _{\Theta} - \lambda I} = \bigvee \{K^{\Theta}_{\lambda} \mathbf{e}_{j}: 1 \leqslant j \leqslant n\}, \qquad |\lambda| < 1$$ and
$$\ran{\widehat{M} _{\Theta} - \lambda I} = \bigvee \{C_{\Theta^{T}} K^{\Theta^{T}}_{1/\overline{\lambda}} \mathbf{e}_{j}: 1 \leqslant j \leqslant n\}, \qquad |\lambda| > 1.$$
An abstract model for $M_{\Theta}$ is then
$$\Gamma(\lambda) = \begin{cases} \sum_{j = 1}^{n} K^{\Theta}_{\lambda} \mathbf{e}_j \otimes \mathbf{e}_j &\mbox{if } |\lambda| < 1, \\
\sum_{j = 1}^{n} C_{\Theta^{T}} K^{\Theta^{T}}_{1/\overline{\lambda}} \mathbf{e}_j \otimes \mathbf{e}_j & \mbox{if } |\lambda| > 1. \end{cases}$$
For any $f \in K_{\Theta}$, we have
$$\langle \Gamma(\lambda)^{*} f, \mathbf{a} \rangle_{\C^n} = \begin{cases} \langle f(\lambda), \mathbf{a} \rangle_{\C^n} &\mbox{if } |\lambda| < 1, \\
\langle \mathbf{a}, \Theta^{T}(1/\overline{\lambda}) (\dagger f)(1/\overline{\lambda})\rangle_{\C^n} & \mbox{if } |\lambda| > 1. \end{cases}$$
\end{Example}

\begin{Example}[de Branges-Rovnyak spaces]
Let $b \in H^{\infty}_{1}$.
As discussed previously, there is a natural
conjugation $C_b : \HH (b) \rightarrow \HH (b)$ which intertwines $M_b$ and $M_b ^*$ and operates on reproducing kernels by
$$ C_b k_\lambda ^b (z) = \frac{b(z) - b(\lambda)}{z-\lambda}.$$ Using this conjugation, Example \ref{scalarmodel} generalizes almost verbatim to this case. In particular
$$ \ran{\widehat{M} _b - \lambda I } ^\perp = \C k_\lambda ^b, \quad \mbox{if} \ |\lambda| < 1, $$
$$ \ran{\widehat{M} _b - \lambda I } = \C C_b k ^b _{1 / \overline{\lambda} }, \quad \mbox{if} \ | \lambda | > 1, $$
and we can define a model for the partial isometry $M_b  $ via
$$ \Gamma (\lambda ) = \left\{ \begin{array}{cc} k _\lambda ^b \otimes 1 & \mbox{if} \ |\lambda | < 1, \\ C_b k ^b _{1 / \overline{\lambda} } \otimes 1 & \mbox{if} \ |\lambda | > 1. \end{array} \right. $$
As before we get that $\HH _{M_b  } = \{ \Gamma ^* f | \ f \in \K _b \}$ where
$$ \Gamma (\lambda ) ^* f = \left\{ \begin{array}{cc} f (\lambda ) & \mbox{if} \ |\lambda | < 1, \\ \overline{(C_b f) (1 / \overline{\lambda} ) } & \mbox{if} \ |\lambda | > 1 .\end{array} \right. $$

\end{Example}

\begin{Example}[$S^{*} \oplus S$]\label{SsPS}
Consider the operator $A := S^{*} \oplus S$
 on $\H := H^2 \oplus H^2$ discussed earlier in Example \ref{12384u234}.  Since $S$ has indices $(0, 1)$ and $S^{*}$ has indices $(1, 0)$, it follows that $A$ has indices $(1, 1)$. Moreover, one can show that  $A$ is completely non-unitary. A calculation using the fact that
$$\ker{A} = \ker{S^{*}} \oplus \ker{S} = \C \oplus \mathbf{0}$$ shows that
$$\TRng(A - \lambda I) = \{0\} \oplus \ran{S - \lambda I} = \{0\} \oplus \C c_{\lambda}, \quad |\lambda| < 1$$ and
$$\TRng(A - \lambda I) = (S^{*} - \lambda I) \ker{S^{*}}^{\perp} = \C c_{1/\lambda} \oplus \{0\}, \quad |\lambda| > 1.$$
In the above,
$$c_{\lambda}(z) = \frac{1}{1 - \overline{\lambda} z}$$ is
the Cauchy kernel (the reproducing kernel for $H^2$).
Thus the model becomes
$$\Gamma(\lambda) = \begin{cases} (0 \oplus c_{\lambda}) \otimes 1 &\mbox{if } |\lambda| < 1, \\
(c_{1/\lambda} \oplus 0) \otimes 1 & \mbox{if } |\lambda| > 1, \end{cases}$$ and so
$$\Gamma(\lambda)^{*} (f_1 \oplus f_2) = \begin{cases} f_2(\lambda) &\mbox{if } |\lambda| < 1, \\
f_{1}(1/\lambda) & \mbox{if } |\lambda| > 1. \end{cases}$$
\end{Example}

\begin{Example}[A restriction of $(S^{*} \oplus S)$]
\label{SsPS2}
Let
$$B := A|_{H^{2} \oplus H^{2}_{0}} = S^{*} \oplus S|_{H^{2}_{0}}$$
which was discussed earlier in Example \ref{12384u234}.
Observe that $\ker{B} = \C (1 \oplus 0)$ and so $\ker{B}^{\perp} = H^{2}_{0} \oplus H^{2}$. Furthermore, $\ran{B} ^\perp =\C (0 \oplus z )$.  One can check that $B \in \VV _1 (H^2 \oplus H^2 _0)$.
For $|\lambda| > 1$ we have
\begin{align*}
\TRng(B - \lambda I)^{\perp} & = \TRng(S^{*} - \lambda I)^{\perp} \oplus \TRng(S - \lambda)^{\perp}
 = \C c_{1/\lambda} \oplus \{0\}.
\end{align*}
When $|\lambda| < 1$ we have
\begin{align*}
\TRng(B - \lambda I)  = \TRng(S^{*} - \lambda I) \oplus \TRng(S|_{H^{2}_{0}} - \lambda I)^{\perp}
 = \{0\} \oplus \ran{S|_{H^{2}_{0}} - \lambda I}^{\perp}.
\end{align*}
A little exercise shows, still assuming $|\lambda| < 1$, that
$$\ran{S|_{H^{2}_{0}} - \lambda I}^{\perp} = \C \frac{c_{\lambda} - c_{0}}{z}.$$
Thus the abstract model for $B$ on $ H^{2} \oplus H^{2}_{0}$ is
$$\Gamma(\lambda) = \begin{cases} (0 \oplus \frac{c_{\lambda} - c_0}{z}) \otimes 1 &\mbox{if } |\lambda| < 1 ,\\
(c_{1/\lambda} \oplus 0) \otimes 1& \mbox{if } |\lambda| > 1. \end{cases}$$
For $f_1 \in H^2$ and $f_2 \in H^{2}_{0}$, we have
$$\Gamma(\lambda)^{*} (f_1 \oplus f_2) = \begin{cases} f_{2}(\lambda)/\lambda &\mbox{if } |\lambda| < 1 ,\\
f_{1}(1/\lambda) & \mbox{if } |\lambda| > 1. \end{cases}$$
\end{Example}

\begin{Example}[$S \otimes S^{*}$]

\label{tpeg2}

Recall the representation of $S \otimes S^*$ as a block operator matrix with $S^*$ repeated on the subdiagonal acting on the Hilbert space
$$\H = \bigoplus _{k \geqslant 0} H^2.$$ One can show that
$$ \ker{S \otimes S^* } = \{ \delta ^{(k)} \} _{k = 0} ^\infty \quad \quad \delta _j ^{(k)} = \delta _{kj} 1 $$ and
$$ \ran {S \otimes S^* } = \{ \widehat{b}  ^{(k)} \} _{k=0} ^\infty \quad \quad \widehat{b} ^{(k)} _j = z^k \delta _{0j}. $$
We also need to calculate $\TRng (S \otimes S ^* - w I  ) ^\perp$.  For any $w \in \C \setminus \T$:
\begin{align*}
 \TRng (S \otimes S ^* - w I  ) ^\perp &= \ran{ (S \otimes S^* - w I ) P_0 } ^\perp\\
 & = \ker{ P_0 (S^* \otimes S - \overline{w} I ) },
 \end{align*} where $P_0 $ projects onto
 $$\H _0 = \bigoplus_{k \geqslant 1} H^2 _0 = \ker{S \otimes S^* } ^\perp.$$
Hence we need to determine the set of all $\mathbf{h} \in \H$ such that
$$ (S^* \otimes S - \overline{w} I ) \mathbf{h} = \begin{bmatrix} -\overline{w}  & S &  &  &  \\  & -\overline{w}  & S &  & \\ & & -\overline{w} &S  &  \\ & & & \ddots & \ddots  \\ & & & & \end{bmatrix} \begin{bmatrix} h_0 \\ h_1 \\ h_2 \\ \vdots \\ $ $  \end{bmatrix} =
\begin{bmatrix} c_0 1  \\ c_1 1 \\ c_2 1 \\ \vdots \\ $ $  \end{bmatrix}, $$ where $c_k \in \C$. This yields the recurrence relation $S h_{k+1} = \overline{w} h_k + c_k 1 $, and acting on both sides of this equation with $S^*$
yields $$ h_{k+1} = \overline{w} S^* h_k, \quad k \in \N \cup \{0 \}. $$ It follows that a basis for $ \TRng (S \otimes S ^* - w I  ) ^\perp $ is the set $\{ \mathbf{h} _k (w) \}$ where
$$ \mathbf{h} _k (w)  = ( z^k , \overline{w} z^{k-1} , ... , \overline{w} ^k 1 , 0 , ... ).$$ Although this is true for any $w \in \C \setminus \T$, in the case where $w \in \D_{e} = \C \setminus \D ^{-}$, we  instead choose
$$ \mathbf{g} _k (w) = \overline{w} ^{-k} \mathbf{h} _k (w) = \big( (z/\overline{w} ) ^k , (z / \overline{w} ) ^{k-1} , ... ,1 , 0 , ... \big) $$ as a basis for $\TRng (S \otimes S^* - w I ) ^\perp$.

A natural choice of model for $S \otimes S^*$ is then
$$ \Gamma (w) : =  \sum _{j =0 } ^\infty \gamma _j (w) \otimes \mathbf{e}_k, $$ where $\{ \mathbf{e}_k \}_{k \geqslant 0}$ is an orthonormal basis for $\C ^\infty := \ell ^2 ( \N \cup \{ 0 \} )$, and
$$ \gamma _j (w) =
\begin{cases}  \mathbf{h} _j (w) &\mbox{if } w \in \D ,\\
\mathbf{g} _j (w) &  \mbox{if } w \in \D_e. \end{cases}$$
\end{Example}

\section{The Liv\v{s}ic characteristic function}

What is a unitary invariant for the partial isometries? We begin with a result of Halmos and McLaughlin \cite{Halmos}.

\begin{Theorem}\label{HML}
Suppose $A, B \in M_{n}(\C)$ are partial isometric matrices with
$$\dim{\ker{A} }= \dim{\ker{B}} = 1.$$Then $A$ is unitarily equivalent to $B$ if and only if their characteristic polynomials coincide.
\end{Theorem}

This theorem breaks down when the defect index is greater than one. Indeed, consider the following matrices:
\begin{equation}\label{ABHML}
A = \left[
\begin{array}{cccc}
 0 & 1 & 0 & 0 \\
 0 & 0 & 1 & 0 \\
 0 & 0 & 0 & 0 \\
 0 & 0 & 0 & 0 \\
\end{array}
\right], \qquad B = \left[
\begin{array}{cccc}
 0 & 1 & 0 & 0 \\
 0 & 0 & 0 & 0 \\
 0 & 0 & 0 & 1 \\
 0 & 0 & 0 & 0 \\
\end{array}
\right].
\end{equation}
From Proposition \ref{canonicalformPI}, we see that $A$ and $B$ are partial isometries with
$$\dim{\ker{A}} = \dim{\ker{B}} = 2.$$ Moreover, the characteristic polynomials of $A$ and $B$ are both equal to $z^4$. However, since $A$ and $B$ have different Jordan forms,  $A$ is not unitarily equivalent to $B$.

The replacement for Theorem \ref{HML} when the defect index is greater than one, and which works for general partial isometries on infinite dimensional Hilbert spaces, is due to Liv\v{s}ic \cite{ MR0113143}.
Let $V\in \mathscr{V}_{n}$ and let $\{\mathbf{v}_1, \ldots, \mathbf{v}_n\}$ be an orthonormal basis for $\ker{V}$. Since the deficiency indices of $V$ are equal, we know from Proposition \ref{has-ue} that $V$ has a unitary extension $U$ (in fact many of them). Define the following $n \times n$ matrix
\begin{equation}\label{Liv-Def}
w_{V}(z) = z [\langle (U - z I)^{-1} \mathbf{v}_{j}, \mathbf{v}_{k}\rangle] [\langle (U - z I)^{-1} U\mathbf{v}_{j}, \mathbf{v}_{k}\rangle]^{-1}, \quad z \in \D.
\end{equation}
Liv\v{s}ic showed that $w_V$ is a contractive analytic $M_{n}(\C)$-valued function on $\D$ and that different choices of basis $\{\mathbf{v}_1, \ldots, \mathbf{v}_n\}$ and unitary extension $U$ will change $w_{V}$ by $Q_{1} w_{V} Q_{2}$, where $Q_1, Q_2$ are constant unitary matrices. The function $w_V$, called the {\em Liv\v{s}ic characteristic function}, is a unitary invariant for $\mathscr{V}_{n}$.

\begin{Theorem}[Liv\v{s}ic]\label{LivsicThm}
If $V_1, V_2 \in \mathscr{V}_n$, then $V_{1}$ and $V_2$ are unitarily equivalent if and only if there are constant $n \times n$ unitary matrices $Q_1, Q_2$ such that
\begin{equation}\label{coincide}
w_{V_1}(z) = Q_1 w_{V_2}(z) Q_2 \quad \forall z \in \D.
\end{equation}
\end{Theorem}

Two $M_n (\C)$-valued contractive analytic functions $w_1, w_2$ on $\D$ are said to {\em coincide} if that satisfy \eqref{coincide}.  Liv\v{s}ic also showed that given any contractive, analytic, $M_n(\C)$-valued, function $w$ on $\D$ with $w(0) = 0$, then there is a $V \in \mathscr{V}_n$ such that $w_{V} = w$. One can quickly check that \eqref{coincide} defined an equivalence relation on such matrix-valued functions. In other words, there is a bijection from unitary equivalence classes of partial isometries with indices $(n,n)$ onto the unitary coincidence equivalence classes of contractive analytic $M_n (\C)$-valued analytic functions on $\D$ which vanish at zero.

Using the definition above to compute $w_V$ can be difficult. However, if one reads Liv\v{s}ic's paper carefully, there is an alternate way of computing $w_V$ \cite{Livsic}.

\begin{Proposition}\label{cooltrick}
For $V \in \mathscr{V}_n$, let $\{g_1, \ldots, g_n\}$ be an orthonormal basis for $\ker{V}$ and let $\{h_1, \ldots h_n\}$ be an orthonormal basis for $\TRng(V)^{\perp}$. For each $z \in \D$ let $\{g_1(z), \ldots, g_n(z)\}$ be a (not necessarily orthonormal) basis for $\TRng(V - z I)^{\perp}$. Then
$$w_{V}(z) = z  [\langle h_j, g_k (z) \rangle]^{-1} [\langle g_j, g_k(z)\rangle].$$
\end{Proposition}

  The construction above can be rephrased as follows.  For any $z \in \C \setminus \T$, let $$ j_z : \C ^n \rightarrow \TRng (V - zI ) ^\perp, $$ be an isomorphism. Furthermore, suppose that $j_0$ is a surjective isometry and let $$ j = j_\infty : \C ^n \rightarrow \ker{V} ^\perp, $$ also be a surjective isometry. The Liv\v{s}ic characteristic function of $V$ is then
  \be w _V (z) = z A(z) ^{-1} B(z), \label{Livchar},  \ee where $A(z) := j_z ^* j_0$ and $B(z) := j_z ^* j_\infty$.


\begin{Example}\label{poiipoipoi}
Suppose $\Theta$ is a scalar-valued inner function with $\Theta(0) = 0$. Note that
$$\ker{S_{\Theta}} = \C \frac{\Theta}{z}, \qquad \ker{S_{\Theta}}^{\perp} = \dom{\widehat{M} _{\Theta}},$$
$$\TRng(S_{\Theta}) = z \dom{\widehat{M} _{\Theta}}, \qquad \TRng(S_{\Theta})^{\perp} = \C.$$ In the formula for $w_V$ in the previous proposition, we get
$g = \Theta/z$ and $h = 1$.
In a similar way we have
$$\TRng(S_{\Theta} - z I) = (w - z) \dom{\widehat{M} _{\Theta} - z I},  \qquad \TRng(S_{\Theta} - z I)^{\perp} = \C k^{\Theta}_{z}.$$
Thus take $g(z) = k^{\Theta}_{z}$ and note that $z \mapsto g(z)$ is anti-analytic. From here, one can show that
$w_{V}(z) = \Theta(z)$.
\end{Example}

\begin{Example}\label{sdkfhjadskfjh}
If $A \in M_{n}(\C)$, $\dim{\ker{A}} = 1$, and
$\sigma(A) = \{0, \lambda_1, \ldots, \lambda_{n - 1}\} \subset \D$, then we know from Proposition \ref{CNU-Mn} that $A \in \mathscr{V}_{1}$. Furthermore, if $\Theta$ is a Blaschke product whose zeros are $\sigma(A)$, then a well-known fact is that $\sigma_{p}(S_{\Theta}) = \sigma(A)$. By Halmos-McLaughlin (Theorem \ref{HML}), $S_{\Theta}$ is unitarily equivalent to $A$ and thus $w_{A} = \Theta$.
\end{Example}

\begin{Example}
For the two matrices $A$ and $B$ from \eqref{ABHML} one can easily compute unitary extensions for $A$ and $B$. Using the definition of $w_A$ and $w_B$ from \eqref{Liv-Def} we get
$$w_{A}(z) = \left[
\begin{array}{cc}
 z & 0 \\
 0 & z^3 \\
\end{array}
\right], \qquad w_{B} = \left[
\begin{array}{cc}
 z^2 & 0 \\
 0 & z^2 \\
\end{array}
\right].$$
From here one can show there are no unitary matrices $Q_1, Q_2$ such that $w_{A}(z) = Q_1 w_{B}(z) Q_2$ for all $z \in \D$.  Indeed, if there were, then $$ |z| = \| w_A (z) \| = \| w_B(z) \|  = |z|^2, \quad z \in \D,$$ which is impossible.
\end{Example}

\begin{Example}
Recall the operator $S^{*} \oplus S$ from Example \ref{SsPS} where we showed that
$$\ker{S^{*} \oplus S} = \C \oplus \{0\} ,\qquad \ker{S^{*} \oplus S}^{\perp} = H_{0}^{2} \oplus H^2.$$
Thus
$$\TRng(S^{*} \oplus S) = S^{*}|_{H_{0}^{2}} \oplus S, \qquad \TRng(S^{*} \oplus S)^{\perp} = \{0\} \oplus \C.$$ In the formula from Proposition \ref{cooltrick} we can take
$g = 1 \oplus 0$ and $h = 0 \oplus 1$.
Notice that
\begin{align*}
\TRng(S^{*} \oplus S - z I) & = (S^{*} - z I) \oplus (S - z I)|_{\ker{S^{*} \oplus S}}\\
& = (S^{*} - z I)|_{H_{0}^{2}} \oplus (S - z I)\\
& = H^2 \oplus (w - z) H^2.
\end{align*}
If $c_{z}(w)$ is the standard Cauchy kernel for $H^2$ we see that
$$\TRng(S^{*} \oplus S - z I)^{\perp} = 0 \oplus c_z.$$
So in Proposition \ref{cooltrick} we can take
$g(z) = 0 \oplus c_z$.
A computation yields
\begin{align*}
w_{S^{*} \oplus S}(z) = z \frac{\langle g, g(z) \rangle}{\langle h, g(z)\rangle}
= z \frac{\langle 1 \oplus 0, 0 \oplus c_z\rangle}{\langle 0 \oplus 1, 0 \oplus c_z\rangle}
 = 0.
\end{align*}

\end{Example}

\begin{Example}
  To calculate the characteristic function of $S \otimes S^*$, it is perhaps easiest to consider the block operator representation from  \eqref{blockop}. In this case,
given $\mathbf{h} = (h_0, h_1 , ... ) \in \H $ where $h_k \in H^2$, we have
$$ (S\otimes S^*) \mathbf{h} = \begin{bmatrix} 0 & &  &  &  \\  S^* & 0 &  &  & \\ & S^* & 0 &  &  \\ & & \ddots & \ddots &  \\ & & & & \end{bmatrix} \begin{bmatrix} h_0 \\ h_1 \\ h_2 \\ \vdots \\ $ $  \end{bmatrix} =
\begin{bmatrix} 0 \\ S^* h_0 \\ S^* h_1 \\ \vdots \\  $ $  \end{bmatrix}, $$ so that an orthonormal basis of $\ker{S \otimes S^*}$ is $\{ \delta ^{(k)} \} _{k=0} ^\infty$
where $\delta ^{k} _j = \delta _{kj} 1. $ Note that
$$\ker{S\otimes S^* } ^\perp = \H _0 := \bigoplus _{k \geqslant 0} H^2 _0,$$ in which $H^2 _0 = \{f \in H^2: f(0) = 0\}$.
Similarly
$$ (S^* \otimes S) \mathbf{h} = \begin{bmatrix} 0 & S &  &  &  \\  & 0 & S &  & \\ & & 0 &S  &  \\ & & & \ddots & \ddots  \\ & & & & \end{bmatrix} \begin{bmatrix} h_0 \\ h_1 \\ h_2 \\ \vdots \\ $ $  \end{bmatrix} =
\begin{bmatrix} Sh_1  \\ S h_2 \\ S h_3 \\ \vdots \\ $ $  \end{bmatrix}, $$ so that
$$\ran{S \otimes S ^* } ^\perp = \ker{S^* \otimes S } = H^2 \bigoplus _{k \geqslant 1} \{ 0 \}$$ has orthonormal basis
$ \{  \widehat{b}  ^{(k)}  \} _{k=0} ^\infty,$ where $\{ b_k \} _{k=0} ^\infty$ is the standard basis of $H^2$, $b_k (z) = z ^k $, and $\widehat{b} ^{(k)} _j = b_k \delta _{0j}$.

A calculation yields
$$ \TRng (S \otimes S ^* - wI ) ^\perp = \{ \mathbf{h} _k (w) \}, \quad \mathbf{h} _k (w)  = ( z^k , \overline{w} z^{k-1} , ... , \overline{w} ^k 1 , 0 , ... ).$$

Putting this together, $w_{S\otimes S^* } (z) = z A(z) ^{-1} B(z)$ where
$$ A(z) = [ \ip{\hat{b} ^{(k)} }{\mathbf{h} _j (z) } ] = [ \ip{z^k}{z^j}_{H^2}  = \delta _{kj} ] = I, $$ and
$$ B(z) = [ \ip{\delta ^{(k)}}{\mathbf{h} _j (z)} ] = [\ip{1}{\overline{z} ^k 1} _{H^2 } \delta _{kj} ] = [ z^k \delta _{kj} ], $$ so that
$$ w_{S\otimes S^*} (z) = \begin{bmatrix} z & & & &  \\ & z^2 & & & \\  &  & z^3 & & \\  & & & \ddots & \\ & & & & \end{bmatrix}. $$

\end{Example}

\section{Herglotz spaces}
\label{Herglotzspace}

There is a canonical choice of abstract model space for operators from $\VV_n$ called a Herglotz space. A $M_{n}(\C)$-valued analytic function on $\D$ is called a {\em Herglotz function} if
$$\Re G(z) := \frac{1}{2} (G(z) + G(z)^{*}) \geqslant 0.$$  There is a bijective correspondence between
$M_n(\C)$-valued Herglotz functions $G$ on $\D$ and $M_n(\C)$-valued contractive analytic functions $b$ on $\D$ given by:
$$ b \mapsto G_b := (I + b)(I - b)^{-1} \quad \mbox{and} \quad G \mapsto b _G := (G - I)(G + I)^{-1}. $$ Any Herglotz function on $\D$  extends to a function on $\C \setminus \T$ by
$$ G (1 / \lambda ) := - G(\overline{\lambda} ) ^*, $$ which ensures that $G$ has non-negative real part on $\C \setminus \T$. Note that if $G = G _b$ for a contractive analytic function $b$, then it follows that $b $ can
be extended to a meromorphic function on $\C \setminus \T$ which obeys $$ b (\lambda )  b (1 / \overline{\lambda} ) ^* =I.$$

Given any contractive analytic $M_n(\C)$-valued function $b$ on $\D$, consider the positive matrix kernel function
$$ K_w (z) := \frac{G_b (z) + G _b (w) ^*}{1 - z \overline{w} }, \quad \quad z,w \in \C \setminus \T. $$
By the abstract theory of reproducing kernel Hilbert spaces \cite{Paulsen-rkhs}, it follows that
there is a reproducing kernel Hilbert space of $\C ^n$-valued functions on $\C \setminus \T$ with reproducing kernel $K_w(z)$. This RKHS is denoted by $\mathscr{L} (b)$ and is
called the
{\em Herglotz space} corresponding to $b$.

\begin{Theorem}
Let $V \in \mathscr{V}_n(\H)$ with characteristic function $w_V = b$ and suppose that $\Gamma$ is an abstract model for $V$. Then there is an isometric multiplier from the abstract model space $\HH_{V, \Gamma}$ onto the Herglotz space $\mathscr{L} (b)$. More precisely, there is an analytic $M_{n}(\C)$-valued function $W$ on $\C \setminus \T$ such that $W \HH_{V, \Gamma} = \mathscr{L}(b)$ and
$$\|W f\|_{\mathscr{L}(b)} = \|f\|_{\HH_{V, \Gamma}}, \quad f \in \HH_{V, \Gamma}.$$
\end{Theorem}

\begin{proof}
 We will prove this by adapting the approach of \cite[Section 4]{AMR} to prove the following formula for the reproducing kernel $k_w ^\Gamma (z)$ of $\HH _{V, \Gamma}$
\begin{equation}\label{absrk}
 k_w ^\Gamma (z)  =  \frac{A (z) A (w ) ^* - z B(z) B(w) ^* \overline{w} }{1 - z \overline{w} },
 \end{equation}
  where $$ A (z) := j_z ^* j_0 \quad \mbox{and} \quad  B(z) := j_z ^* j,$$
and the maps $j_z , j : \C ^n \rightarrow \H $ are as defined before formula (\ref{Livchar}). Namely, $$ j_z : \C ^n \rightarrow \TRng (V - zI ) ^\perp, $$ is an isomorphism such that $z \mapsto j_z$ is
anti-analytic, $j_0$ is a surjective isometry and $$ j = j_\infty : \C ^n \rightarrow \ker{V} ^\perp, $$ is also an onto isometry.

To prove \eqref{absrk}, consider the abstract model space $\HH _{V, \Gamma}$ for the model
$\Gamma$ of $V$. The reproducing kernel is
$$k_w  (z) = \Gamma (z) ^* \Gamma (w).$$ Now for any $\mathbf{u} , \mathbf{v} \in \C ^n$, if $Q_\infty$ denotes the projection of $\H$ onto $\ker{\ZZZ _{V, \Gamma }} ^\perp$ and $Q_0$ denotes the projection of $\H$ onto
$\ran{\ZZZ _{V, \Gamma } }$, then
\begin{align*}
 \left( (\ZZZ _\Gamma  ^* k_w ) (z) \mathbf{u} , \mathbf{v} \right) _{\C ^n} & =  \ip{ \ZZZ _\Gamma  ^* k_w \mathbf{u} }{k_z \mathbf{v} } _\Gamma \nonumber  = \ip{k_w \mathbf{u}}{ \ZZZ _\Gamma  Q_\infty  k_z \mathbf{v} } _\Gamma  \nonumber \\
& =  \overline{\ip{ \ZZZ _\Gamma  Q_\infty k_z \mathbf{v}}{k_w \mathbf{u} } _\Gamma } \nonumber \\
& = \overline{w} \left[  \left( k_w (z) \mathbf{u} , \mathbf{v} \right) _{\C ^n}  - \left( (P_\infty k_w ) (z) \mathbf{u} , \mathbf{v} \right) _{\C ^n } \right], \nonumber
\end{align*}
where $P_\infty = I - Q _\infty $. However, $\ZZZ _\Gamma  ^*$ also acts as multiplication by $1/z$ on $\ran{\ZZZ _\Gamma }$ and so
\begin{align*}
  \left( (\ZZZ _\Gamma  ^* k_w ) (z) \mathbf{u} , \mathbf{v} \right) _{\C ^n} & =  \frac{1}{z} \ip{Q_0  k_w \mathbf{u}}{ k_z \mathbf{v} } _\Gamma \nonumber \\
& =\frac{1}{z} \left[ \left( k_w (z) \mathbf{u} , \mathbf{v} \right) _{\C ^n } - \left( (P_0 k_w ) (z)  \mathbf{u} , \mathbf{v} \right)  _{\C ^n } \right], \nonumber
\end{align*}
where $P_0 = I - Q_0$. Equating these two expressions yields
$$ k_w (z) = \frac{(P_0 k_w ) (z) - z\overline{w} (P_\infty k_w ) (z) }{1-z \overline{w} }. $$ Now define surjective isometries $j _0 : \C ^n \rightarrow \ran{V} ^\perp$ and $j_\infty : \C ^n \rightarrow \ker{V}$,
and let $j_z = \Gamma (z)$ for any $z \in \C \setminus \T$, $z \neq 0$. Observe that if $U_\Gamma : \H \rightarrow \HH _{V, \Gamma}$ is the unitary transformation onto the abstract model space given by
$ U_\Gamma f (z) = \Gamma (z ) ^* f, $ then $$ P _0 = U_\Gamma j_0 j_0 ^* U _\Gamma ^* $$ and $$ k_w \mathbf{u}  = U _\Gamma j _w \mathbf{u} $$  for any $\mathbf{u} \in \C ^n$. It follows that
\begin{align*}
  (P_0 k_w) (z) & =   j_z ^* U_\Gamma ^* U _\Gamma j_0 j_0 ^* U _\Gamma ^* U_\Gamma j_w \nonumber \\
& =  j_z ^* j_0 j_0 ^* j_w  =  A(z) A(w) ^*, \nonumber
\end{align*} and similarly
$$ (P_\infty k_w )(z) = j_z ^* j_\infty j_\infty ^* j_w = B(z) B(w) ^*, $$ so that the reproducing kernel for $\HH _\Gamma $ takes the form in \eqref{absrk} as claimed.

To obtain the isometric multiplier $W$ from $\HH_{V, \Gamma}$ onto $\mathscr{L}(b)$, we proceed as follows.
By Proposition \ref{cooltrick}, it further follows that for $z \in \D$,
$A(z)$ is invertible and $b(z) = z A(z) ^{-1} B(z)$ and so
$$ k_w ^\Gamma (z) = A(z) \left( \frac{1 - b(z) b(w) ^* }{1 - z \overline{w} } \right) A(w) ^*, \quad z,w  \in \D, $$ and similarly for $z, w \in \C \setminus \D ^{-} $, $B(z), B(w)$ are invertible and so
\begin{align*} k_w ^\Gamma (z) & =  B(z) \left( \frac{1 - b (z) ^{-1} (b (w) ^* ) ^{-1} }{1 - z^{-1} \overline{w} ^{-1} } \right) B(w) ^*, \quad \quad  |z | , |w| > 1 \nonumber \\
& =   B(z) \left( \frac{1 - b ( 1 /\overline{z} ) ^*  b (1 / \overline{w} ) }{1 - z^{-1} \overline{w} ^{-1} } \right) B(w) ^*, \quad  \quad  |z | , |w| > 1.
\end{align*}

Now compare this kernel for $\HH _\Gamma$ to that of $\mathscr{L} (b)$, $$K_w (z)  =  \frac{G_b (z) + G_b (w) ^*}{1-z\overline{w}}.$$  Using the formula $$ G_b := \frac{1+b}{1-b}, $$ and the fact that $b(z) = z A(z) ^{-1} B(z)$ for $z \in \D$, we have
$$ G_b (z) = \frac{A(z) + z B(z)}{A(z) - z B(z) } , \quad \quad z \in \C \setminus \T. $$ Inserting this expression into the formula for the kernel $K_w (z)$ of $\mathscr{L} (b)$ yields
\small
$$ K_w (z) = \sqrt{2} (A(z) + z B(z) ) ^{-1} \left( \frac{A(z) A(w) ^* -z B(z) B(w) ^* \overline{w} }{1-z \overline{w} } \right) (A(w) ^* + B(w) ^* \overline{w} ) ^{-1} \sqrt{2}.$$
\normalsize
The preceding simplifies to
$$W(z) k_w ^\Gamma (z) W(w)^*, \quad \quad z,w \in \C \setminus \T$$
where
$$ W(z) := \sqrt{2} \left( A(z) + z B(z) \right) ^{-1}, \quad \quad z \in \C \setminus \T.$$
Hence $W : \HH _{V, \Gamma} \rightarrow \mathscr{L} (b)$ is an isometric multiplier of $\HH _\Gamma$ onto $\mathscr{L} (b)$.

It follows that given any model $\Gamma$ for the partial isometry $V$, we can define a new model:
$$ \wt{\Gamma} (z) := \Gamma (z) W(z) ^*, $$ so that
$$ k _w  ^{\wt{\Gamma}} (z) = W(z) \Gamma (z) ^* \Gamma (w)  W (w) ^* = K_w (z).$$  This shows that $\HH _{\wt{\Gamma} } = \mathscr{L} (b)$, so that $\mathscr{L} (b)$ can be thought of as the canonical
model space for a partial isometry with characteristic function $w _V = b$. Since $\mathscr{L} (b)$ is canonical in this sense, we will use the notation $\ZZZ _b $ for the partial isometry which acts as multiplication by $z$ on its initial space in $\mathscr{L} (b)$ and is unitarily equivalent to $V$.  It is also straightforward to check that the characteristic function of $\ZZZ _b$ is $b$ so that $b = w_V$.
\end{proof}

\section{The partial order of Halmos and McLaughlin}

Halmos and McLaughlin gave the following partial order on the set of all partial isometries, not only the ones in $\VV_{n}$ \cite{Halmos}. For two partial isometries $A, B$, we say that $A \Halmos B$ if $B$ agrees with $A$ on the initial space of $A$.  Since $A^{*} A$ is the orthogonal projection onto its initial space, $A \Halmos B$ if and only if
$$A = B A^{*} A.$$ The following follows quickly from the definition of $\Halmos$.

\begin{Proposition}\label{P-Halmos}
The relation $\Halmos$ defines a partial order on the set of partial isometries.
\end{Proposition}

\begin{Example}
\begin{enumerate}
\item Suppose that $\{\mathbf{u}_1, \mathbf{u}_{2}, \ldots, \mathbf{u}_{n}\}$ is any orthonormal basis for $\C^n$. The matrices
$$A = [\mathbf{u}_1 | \mathbf{u}_2| \cdots| \mathbf{u}_{r}| \mathbf{0}|\mathbf{0}|\cdots|\mathbf{0}], \quad B = [\mathbf{u}_1 | \mathbf{u}_2| \cdots| \mathbf{u}_r| \mathbf{u}_{r + 1}|\mathbf{0}|\cdots|\mathbf{0}]$$
are partial isometric matrices and one can check that $A \Halmos B$.
\item
Consider the $n \times n$ block matrix
$$V = \left[\begin{array}{cc}
 0 & 0 \\
 U & 0 \\
\end{array}\right]$$
where $U$ is any $r \times r$ unitary matrix. If $A$ is any $(n - r) \times (n - r)$ partial isometric matrix, one can show that
$$V_{A} = \left[\begin{array}{cc}
 0 & A \\
 U & 0 \\
\end{array}\right]$$
is a partial isometry. Using block multiplication of matrices one can verify the formula
$$V = V_{A} (V^{*} V)$$ and so $V  \Halmos V_A$. One can argue that if $W$ is any partial isometry with $V \Halmos W$ then $W = V_{A}$ for some partial isometry $A$. We thank Yi Guo and Zezhong Chen for pointing this out to us.
\item
Recall the operators $A = S^{*} \oplus S$ on $H^2 \oplus H^2$ and $B = A|_{H^2 \oplus H^{2}_{0}}$ from Examples \ref{SsPS} and \ref{SsPS2}. Since $\ker{B} = \C \oplus \mathbf{0}$,  $\ker{B}^{\perp} = H^{2}_{0} \oplus H_{0}^{2}$. Note that
$$B|_{\ker{B}^{\perp}} = A|_{\ker{B}^{\perp}}$$ and so $B \Halmos A$.
\end{enumerate}
\end{Example}

\section{Two other partial orders}

Let $\SS_{n}(\H)$ denote the simple symmetric linear transformations on $\H$ with $(n, n)$ deficiency indices and $\SS_{n}$ denote the collection of all such operators on any Hilbert space. Recall here that a symmetric linear transformation $S$ is said to be {\em simple} if its Cayley transform $V = \beta(S)$ is completely non-unitary.

\begin{Definition}\label{D-1a}\hfill
\begin{enumerate}
\item For $S \in \SS_{n}(\H_{S})$ and $T \in \SS_{n}(\H_{T})$ we say that $S \unitary T$ if there exists an isometric map $U: \H_{S} \to \H_T$ such that $U\dom{S} \subset \dom{T}$ and
$$U S|_{\dom{S}} = T U|_{\dom{S}}.$$
\item For $S \in \SS_{n}(\H_{S})$ and $T \in \SS_{n}(\H_{T})$ we say that $S \quasi T$ if there exists a bounded injective map $X: \H_{S} \to \H_T$ such that $X\dom{S} \subset \dom{T}$ and
$$X S|_{\dom{S}} = T X|_{\dom{S}}.$$
\end{enumerate}
\end{Definition}

\begin{Definition}\label{D-2a}\hfill
\begin{enumerate}
\item For $A \in \VV_{n}(\H_{A})$ and $B \in \VV_{n}(\HH_B)$ we say that $A \unitary B$ if there exists an isometric map $U: \HH_{A} \to \HH_{B}$ such that $U \ker{A}^{\perp} \subset \ker{B}^{\perp}$ and
$$U A|_{\ker{A}^{\perp}} = B U|_{\ker{A}^{\perp}}.$$
\item For $A \in \VV_{n}(\H_{A})$ and $B \in \VV_{n}(\HH_B)$ we say that $A \quasi B$ if there exists an injective $X: \HH_{A} \to \HH_{B}$ such that $X \ker{A}^{\perp} \subset \ker{B}^{\perp}$ and
$$X A|_{\ker{A}^{\perp}} = B X|_{\ker{A}^{\perp}}.$$
\end{enumerate}
\end{Definition}

Given $A \in \VV_n(\H_{A})$ and $B \in \VV_{n}(\H_{B})$, let
$S = \beta^{-1}(A)$ and $T = \beta^{-1}(B)$, where
$$\beta(z) = \frac{z - i}{z + i}, \quad \beta^{-1}(z) = i \frac{1 + z}{1 - z}.$$
Standard theory implies that
$S \in \SS_{n}(\H_{A})$ and $T \in \SS_{n}(\H_{B})$. The following two facts are straightforward to verify.

\begin{Proposition}
With the above notation we have
\begin{enumerate}
\item $A \unitary B \iff S \unitary T$;
\item $A \quasi B \iff S \quasi T$.
\end{enumerate}
\end{Proposition}

\begin{Proposition}
The relations $\unitary$ and $\quasi$ are reflexive and transitive.
\end{Proposition}

Recall here that any binary relation which is reflexive and transitive is called a pre-order \cite[Definition 5.2.2]{BS03}.
This proposition shows that $\unitary$ and $\quasi$ are pre-orders on $\VV$, the set of all
completely non-unitary  partial isometries with equal deficiency indices.

\begin{Definition}
For $A, B \in \VV_n$, we say that
\begin{enumerate}
\item $A \sim B$ if both $A \unitary B$ and $B \unitary A$;
\item $A \sim_{q} B$ if both $A \quasi B$ and $B \quasi A$.
\end{enumerate}
\end{Definition}

It is well-known that given any pre-order $\lesssim$ on a set $S$, if one defines a binary relation $\sim$ on $S \times S$ as above, then
$\sim$ is an equivalence relation and $\lesssim$ can be viewed as a partial order on $S / \sim $ \cite[Proposition 5.2.4]{BS03}. In particular we
have that:

\begin{Corollary}
 The binary relations $\sim$ and $\sim_q$ are equivalence relations on $\VV_n$ and the pre-orders $\unitary$ and $\quasi$
induce partial orders on $\VV_n/\sim$ and $\VV_n/\sim_q$ respectively.
\end{Corollary}

At this point, one could ask what the equivalences classes generated by $\sim$ and $\sim_{q}$ are. In particular, one might expect that the equivalence classes of $\sim$ to simply be unitary equivalence classes. We can show that this is the case for a large subclass of $\VV _n$ (see Theorems \ref{eclass} and \ref{eclass2}), but the proofs are nontrivial.  Before investigating the nature of these equivalence classes further, it will first be convenient to develop a function theoretic characterization of these two partial orders in terms of multipliers between the abstract model spaces or Herglotz spaces associated with partial isometries in $\VV _n$.

\section{Partial orders and multipliers}

\label{POmult}

Recall the associated operator $\widehat{\ZZZ} _A$ of multiplication by the independent variable  on $\HH_A$ defined on $\dom{\widehat{\ZZZ} _A} = \{f \in \HH_A: z f \in \HH_A\}$. Also recall the associated partial isometry $\ZZZ_{A}$ obtained by extending $\widehat{\ZZZ} _{A}$ by zero on $\dom{\widehat{\ZZZ} _{A}}^{\perp}$.

\begin{Theorem}
For $A, B \in \VV_n$ with associated operators $\ZZZ_{A}$ on $\HH_A$ and $\ZZZ_{B}$ on $\HH_B$, the following are equivalent
\begin{enumerate}
\item $A \quasi B$;
\item $\ZZZ_{A} \quasi \ZZZ_{B}$;
\item There exists a multiplier from $\HH_A$ to $\HH_B$.
\end{enumerate}
\label{multiplier}
\end{Theorem}

\begin{proof}
Assume that $A \quasi B$. Then there is a bounded injective operator $X: \H_{A} \to \H_{B}$ with
$X \ker{A}^{\perp} \subset \ker{B}^{\perp}$ and $X A|_{\ker{A}^{\perp}} = B X|_{\ker{A}^{\perp}}$. Let $U_{A}: \H_{A} \to \HH_{A}$ be the unitary which induces the unitary equivalence between $A$ and $\ZZZ_{A}$, that is, $U_{A} \ker{A}^{\perp} = \dom{\widehat{\ZZZ} _{A}} = \ker{\ZZZ_{A}}^{\perp}$ (and hence $U_{A} \ker{A}  = \ker{\ZZZ_{A}}$) and such that
$\ZZZ_{A} U_{A} = U_{A} A$. Define
$$Y: \HH_{A} \to \HH_{B}, \quad Y = U_{B} X U_{A}^{*}$$ and note that
\begin{align*}
Y \ker{\ZZZ_{A}}^{\perp} & = U_{B} X U_{A}^{*} \ker{\ZZZ_{A}}^{\perp} = U_{B} X \ker{A}^{\perp}\\
& \subset U_{B} \ker{B}^{\perp} =\ker{\ZZZ_{B}}^{\perp}.
\end{align*}
Also note that if $f \in \ker{\ZZZ_{A}}^{\perp} = \dom{\widehat{\ZZZ} _{A}}$ then $f = U_{A} x_{f}$ for some $x_{f} \in \ker{A}^{\perp}$. Moreover,
\begin{align*}
Y \widehat{\ZZZ} _{A} f & = U_{B} X U_{A}^{*} \widehat{\ZZZ} _{A} U_{A} x_{f}
 = U_{B} X A x_{f}
 = U_{B} B X x_{f}\\
& = \widehat{\ZZZ} _{B} U_{B} X x_{f} = \widehat{\ZZZ} _{B} U_{B} X U_{A}^{*} f = \widehat{\ZZZ} _{B} Y f.
\end{align*}
This is precisely the definition of $\ZZZ_{A}\quasi \ZZZ_{B}$. Thus statement (1) implies statement (2).

The proof of $(2) \implies (1)$ is similar. Indeed, if $\ZZZ_{A} \quasi \ZZZ_{B}$ then there is a bounded injective operator $X_{1}: \HH_{A} \to \HH_{B}$ with $X_{1} \ker{\ZZZ_{A}} ^\perp \subset \ker{\ZZZ_{B}}^{\perp}$ such that
$X_{1} \ZZZ_{A}|_{\ker{\ZZZ_{A}}^{\perp}}  = \ZZZ_{B} X_{1}|_{\ker{\ZZZ_{A}}^{\perp}}.$ Define
$$Y_{1}: \H_{A} \to \H_{B}, \quad Y_1 = U_{B}^{*} X_1 U_{A}$$ and follow the computation above to show that
$Y_{1} A|_{\ker{A}^{\perp}} = B Y_{1}|_{\ker{A}^{\perp}}$.

We now show the equivalence of statements (2) and (3). First we note that for any $\mathbf{a} \in \C$ and $w \in \C \setminus \T$ that
$$k^{A}_{w} \mathbf{a} \in \ran{\widehat{\ZZZ} _{A} - w I}^{\perp}, \quad k^{B}_{w} \mathbf{b} \in \ran{\widehat{\ZZZ} _{B} - w I}^{\perp}.$$  This implies that
$$\ran{\widehat{\ZZZ} _{A} - w I}^{\perp} = \bigvee_{\mathbf{a} \in \C^n} k_{w}^{A} \mathbf{a}, \quad \ran{\widehat{\ZZZ} _{B} - w I}^{\perp} = \bigvee_{\mathbf{a} \in \C^n} k_{w}^{B} \mathbf{a}.$$

Recall that $$\ker{\ZZZ_{A}}^{\perp} = \dom{\widehat{\ZZZ} _{A}}, \quad \ker{\ZZZ_{B}}^{\perp} = \dom{\widehat{\ZZZ} _{B}}.$$
Now suppose that $X: \HH_{A} \to \HH_B$ is injective with
$$X \ker{\ZZZ_{A}}^{\perp} \subset \ker{\ZZZ_{B}}^{\perp}$$ and
$$X \ZZZ_{A}|_{\ker{\ZZZ_{A}}^{\perp}} = \ZZZ_{B} X|_{\ker{\ZZZ_A}^{\perp}}.$$
Then for any $f \in \HH_A$ we have
$$\langle (X f)(z), \mathbf{a}\rangle_{\C^n}  = \langle X f, k_{z}^{A} \mathbf{a}\rangle = \langle f, X^{*} k_{z}^{A} \mathbf{a}\rangle.$$
Since
$$X \ran{\widehat{\ZZZ} _{A} - w I} \subset \ran{\widehat{\ZZZ} _B - w I}, \quad w \in \C \setminus \T,$$ we obtain
$$X^{*} \ran{\widehat{\ZZZ} _{B} - w I}^{\perp} \subset \ran{\widehat{\ZZZ} _{A} - w I}^{\perp} = \bigvee_{\mathbf{a} \in \C^n} k_{w}^{A} \mathbf{a}.$$
Thus
$$X^{*} k_{z}^{B} \mathbf{a} = k_{z}^{A} R(z) \mathbf{a}$$ for some $R(z)^{*} \in M_{n}(\C)$, which says that
\begin{align*}
\langle (X f)(z), \mathbf{a}\rangle_{\C^n}  &= \langle f, X^{*} k_{z}^{B} \mathbf{a}\rangle
 = \langle f, k_{z}^{A} R(z)^{*} \mathbf{a}\rangle_{\HH_{A}}\\
& = \langle f(z), R(z)^{*} \mathbf{a}\rangle_{\C^n}
 = \langle R(z) f(z), \mathbf{a}\rangle_{\C^n}.
\end{align*}
This says that
$$(X f)(z) = R(z) f(z), \qquad z \in \C \setminus \T.$$

Conversely, suppose that $R$ is a multiplier from $\HH_{A}$ to $\HH_B$. Then, via the closed graph theorem, $M_{R}$, multiplication by $R$, is an injective bounded operator from $\HH_{A}$ to $\HH_B$.
This means that if $f \in \ker{\ZZZ_{A}'}^{\perp} = \dom{\widehat{\ZZZ} _{A}}$
 then $R f,$ and $z R f = R z f \in \HH_{B}$ and so $R f \in \dom{\widehat{\ZZZ} _{B}} = \ker{\ZZZ _{B}}^{\perp}$. Thus
 $$M_{R} \ker{\ZZZ_{A}}^{\perp} \subset \ker{\ZZZ_{B}}^{\perp}.$$
Furthermore, for $f \in \ker{\ZZZ_{A}}^{\perp}$ we have
\begin{align*}
(M_{R} \ZZZ_{A} f)(z) & =
(R \widehat{\ZZZ}_{A} f)(z) = R(z) z f(z) = z R(z) f(z)\\
& = (\widehat{\ZZZ} _{B} R f)(z) = (\ZZZ_{B} R f)(z).
\end{align*}
Thus $\ZZZ_{A} \quasi \ZZZ_{B}$.
\end{proof}

\begin{Theorem}
For $A, B \in \VV_n$ with associated operators $\ZZZ_{A}$ on $\HH_A$ and $\ZZZ_{B}$ on $\HH_B$, the following are equivalent
\begin{enumerate}
\item $A \unitary B$;
\item $\ZZZ_{A} \unitary \ZZZ_{B}$;
\item There exists an isometric multiplier from $\HH_A$ to $\HH_B$.
\end{enumerate}
\label{isomult}
\end{Theorem}

\begin{proof}
The proof is the same as before but multiplication by $R$ is an isometric multiplier.
\end{proof}

\begin{Example} \label{scalar}
Suppose $\Theta_1$ and $\Theta_2$ are inner functions and consider the partial isometries $M_{\Theta_1}$ on $\K_{\Theta_1}$ and $M_{\Theta_2}$ on $\K_{\Theta_2}$. If $\Theta_1$ divides $\Theta_2$, i.e., $\Theta_{1}^{-1} \Theta_{2}$ is an inner function, then $M_{\Theta_1} \unitary M_{\Theta_2}$ since $\K_{\Theta_1} \subset K_{\Theta_2}$. The isometric $U: \K_{\Theta_1} \to K_{\Theta_2}$ can be taken to be the inclusion operator. Note that the norm on both spaces in the same (the $H^2$ norm) and so this inclusion is indeed isometric.
\end{Example}

\begin{Example} \label{inner}
The previous example can be generalized further. Suppose that $\Theta _2$ is an arbitrary contractive $M_n$-valued analytic function such that $\Theta_2 = \Theta _1 \Phi$ where $\Phi$ is a contractive analytic $M _n$-valued function and $\Theta _1$ is inner. Then by \cite[II-6]{Sarason}, $\K_{\Theta _1}$ is contained isometrically
in the deBranges-Rovnyak space $\HH(\Theta _2)$.  By Section \ref{Herglotzspace}, the reproducing kernel for any Herglotz space $\mathscr{L} (\Theta )$ on $\C \setminus \T$ can be expressed as
\begin{align*}
K_w ^{\Theta  } (z) & = \sqrt{2} (I-\Theta (z) ) ^{-1} \left(  \frac{I - \Theta (z) \Theta (w) ^* }{1 - z \overline{w} }  \right) \sqrt{2} (I - \Theta (w) ^* ) ^{-1}\\
& = V(z) k_w ^\Theta (z) V(w) ^*,
\end{align*}
where $k_w ^\Theta (z)$ is the reproducing kernel for the deBranges-Rovnyak space $\HH (\Theta)$. It follows that multiplication by $V (z) := \sqrt{2} (1-\Theta _2 (z) ) ^{-1}$ is an isometry from $\HH(\Theta _2)$ into the Herglotz space $\mathscr{L} (\Theta _2 )$. Hence $V : \K _{\Theta _1} \subset \HH (\Theta _2 ) \rightarrow \mathscr{L} (\Theta _2 )$, the operator of multiplication by $V(z)$, is an isometry of $\K _{\Theta _1} $ into $\mathscr{L} (\Theta _2 )$. Recall that the canonical partial isometry which acts as multiplication by $z$ on the largest possible domain in $\mathscr{L} (\Theta _2 )$ is denoted by $\ZZZ _{\Theta _2} $ (see Section \ref{Herglotzspace}), and the corresponding isometric linear transformation is $\widehat{\ZZZ} _{\Theta _2}$. By the definition of the domain of $\dom{\widehat{M} _{\Theta _1}}$, $$ \dom{\widehat{M} _{\Theta _1} } = \ker{ M_{\Theta _1}  } ^\perp = \{ f \in \K _{\Theta _1} | \ zf(z) \in \K _{\Theta _1} \}. $$ It follows that $V \dom{\widehat{M} _{\Theta _1}} \subset \dom{\widehat{\ZZZ} _{\Theta _2} } $ and that $V \widehat{M} _{\Theta _1} V ^* \subset \widehat{\ZZZ} _{\Theta _2} $ so that
$ M_{\Theta _1}  \unitary \ZZZ _{\Theta _2} $. Since $\ZZZ _{\Theta _1}  \cong M _{\Theta _1 } $, this also shows that $\ZZZ _{\Theta _1}  \unitary \ZZZ _{\Theta _2} $ whenever $\Theta _1$ is inner, $\Theta _2$ is contractive and $\Theta _1 \leqslant \Theta _2$.

Now suppose that $\Phi := \Theta \Gamma$ where $\Phi, \Theta, \Gamma$ are all scalar-valued inner functions on $\D$.  Let
$$ \Lambda := \left( \begin{array}{cc} \Theta & 0 \\ 0 & \Gamma \end{array} \right). $$ Then $\Lambda$ is a $2\times 2$ matrix-valued inner function,
and note that $M _\Lambda '$ has indices $(2,2)$, and that there is a natural unitary map $W$ from $\K _\Lambda = \K _\Theta \oplus \K _\Gamma$ onto
$\K _\Phi = \K _\Theta \oplus \Theta \K _\Gamma$. Namely $$ W (f \oplus g ) := f + \Theta g, $$ so that if we view elements of $\K _\Lambda $ as column vectors then $W$ acts as multiplication by the $1 \times 2$
matrix function $$ W (z) = (1 , \Theta (z) ).$$ It follows that $M _\Lambda  \unitary M_\Phi $, where $M_\Lambda $ has indices $(2,2)$ and
$M _\Phi $ has indices $(1,1)$.
\end{Example}

\begin{Example} \label{divisor}
    Even more generally suppose that $\Theta , \Phi$ are arbitrary contractive analytic $M_n (\C )-$valued functions on $\D$ such that $\Theta$ divides $\Phi$, i.e., $\Phi = \Theta \Gamma$ for some other
contractive analytic $M_n (\C)$-valued function $\Gamma$ on $\D$. As in the previous example the reproducing kernel for the Herglotz space $\mathscr{L} (\Theta )$ on $\C \setminus \T$ is:
$$ K_w ^\Theta (z) = \frac{G_\Theta (z) + G_\Theta (w) ^*}{1- z\overline{w}}, $$ and using that $G_\Theta = (1 +\Theta ) (1 - \Theta ) ^{-1}$, this can be re-expressed as
$$ K_w ^\Theta (z) =  \sqrt{2} (I-\Theta (z) ) ^{-1} \left(  \frac{I - \Theta (z) \Theta (w) ^* }{1 - z \overline{w} }  \right) \sqrt{2} (I - \Theta (w) ^* ) ^{-1}. $$
Recall here from Section \ref{Herglotzspace} that $\Theta$ is extended to a matrix function on $\C \setminus \T$ using the definition $\Theta (z) \Theta (1 / \overline{z} ) ^* = I $.
Let $$ W (z) := (I - \Phi (z) ) ^{-1} (I-\Theta (z) ), $$ and observe that
$$ K_w ^\Phi (z) - W(z) K_w ^\Theta (z) W(w) ^*$$ is equal to
\ba  \sqrt{2} (I - \Phi (z) ) ^{-1} \Theta (z)  \left( \frac{ I - \Gamma (z) \Gamma (w) ^* }{1 - z \overline{w}}  \right) \Theta (w) ^* (I-\Phi (w) ^* )^{-1} \sqrt{2} \nonumber \\
 = (I - \Phi (z) ) ^{-1} \Theta (z) (I - \Gamma (z) ) K_w ^\Gamma (z) (I-\Gamma (w) ) ^* \Theta (w) ^* (I-\Phi (w) ^* )^{-1}, \nonumber \ea
where $K_w ^\Gamma (z)$ is the reproducing kernel for the Herglotz space $\mathscr{L} (\Gamma )$ on $\C \setminus \T$.  This shows that the difference $K_w ^\Phi (z) - W(z) K_w ^\Theta (z) W(w) ^*$ is a positive $M_n (\C )$-valued kernel function on $\C \setminus \T$, and so it follows from the general theory of reproducing kernel Hilbert spaces that $W(z)$ is a contractive multiplier of $\mathscr{L} (\Theta )$ into $\mathscr{L} (\Phi )$ \cite[Theorem 10.20]{Paulsen-rkhs}. Theorem \ref{multiplier} now implies that $\ZZZ _\Theta \quasi \ZZZ _\Phi$ whenever $\Theta$ divides $\Phi$.
\end{Example}

\begin{Example} \label{FShift}
Suppose that $\Theta$ is a scalar inner function and $a \in \D$. A theorem of Crofoot \cite{MR1313454} says that the operator
$$U: \K_{\Theta} \to \K_{\Theta_a}, \quad U f = \frac{\sqrt{1 - |a|^2}}{1 - \overline{a} \Theta} f$$ is a unitary operator from $\K_{\Theta}$ onto $\K_{\Theta_{a}}$, where
$$\Theta_{a} = \frac{\Theta - a}{1 - \overline{a} \Theta}.$$ Thus $M_{\Theta} \cong M_{\Theta_a}$ and so certainly $M_{\Theta} \unitary M_{\Theta_a}$.
\end{Example}

\begin{Example}
Continuing the previous example, now suppose that $\Phi$ is a scalar inner function such that $\Theta_a$ divides $\Phi$.  Then
 one can see (by composing the unitary operators from the previous two examples) that $M_{\Theta} \unitary M_{\Phi}$.
\end{Example}

\begin{Example}
For a scalar inner function $\Theta$, let $\sigma$ be the unique finite positive measure on $\T$ satisfying
$$\frac{1 - |\Theta(z)|^2}{|1 - \Theta(z)|^2} = \int_{\T} \frac{1 - |z|^2}{|\zeta - z|^2} d\sigma(\zeta).$$
Such a measure $\sigma$ is one of the {\em Clark measures corresponding to $\Theta$}. From Clark theory \cite{Sarason} we know that
$$\K_{\Theta} = (1 - \Theta) \mathscr{C}_{\sigma} L^2(\sigma),$$
where
$$\mathscr{C}_{\sigma}: L^2(\sigma) \to \mathscr{O}(\D), \quad (\mathscr{C}_{\sigma} f)(z) = \int_{\T} \frac{f(\zeta)}{1 - \overline{\zeta} z} \,d \sigma(\zeta),$$
 is the Cauchy transform operator, and
$$\|(1 - \Theta) \mathscr{C}_{\sigma} f\| = \|f\|_{L^2(\sigma)}.$$
It is also known that
$E := \{\zeta \in \T: \lim_{r \to 1^{-}} \Theta(r \zeta) = 1\}$ is a carrier for $\sigma$. Let $F \subset E$ be such that $\mu = \sigma|_{F}$ is not the zero measure. Standard Clark theory says that $\mu$ is the Clark measure for some inner function $\Phi$, meaning that
$$\frac{1 - |\Phi(z)|^2}{|1 - \Phi(z)|^2} = \int_{\T} \frac{1 - |z|^2}{|\zeta - z|^2} d \mu(\zeta), \qquad z \in \D.$$

Note that we have $L^2(\mu) \subset L^2(\sigma)$ (understanding this inclusion by extending the functions in $L^2(\mu)$ to be zero on $\T \setminus F$). Furthermore, observe that
\begin{align*}
\K_{\Theta}  = (1 - \Theta) \mathscr{C}_{\sigma} L^2(\sigma) \supset (1 - \Theta) \mathscr{C}_{\sigma} L^2(\mu)
 = \frac{1 - \Theta}{1 - \Phi} (1 - \Phi) \mathscr{C}_{\sigma} L^2(\mu) = \frac{1 - \Theta}{1 - \Phi} \K_{\Phi}.
\end{align*}
Thus $(1 - \Theta)(1 - \Phi)^{-1}$ is a multiplier from $\mathcal{K}_{\Phi}$ to $\mathcal{K}_{\Theta}$. Furthermore,
if $F \in \mathcal{K}_{\Phi}$, then $F = (1 - \Phi) \mathscr{C}_{\sigma} f$, where $f \in L^2(\mu)$ (and considered to also be an element of $L^2(\sigma)$ by defining it to be zero on $\T \setminus F$). Finally,
\begin{align*}
\left\|\frac{1 - \Theta}{1 - \Phi} F\right\|^{2}  = \|(1 - \Theta) \mathscr{C}_{\sigma} f\|^{2}
 = \|f\|^{2}_{L^2(\sigma)}
 = \int |f|^2 d \sigma
 = \int |f|^2 d \mu
 = \|F\|^{2}.
\end{align*}
The last equality says that $(1 - \Theta)(1 - \Phi)^{-1}$ is an isometric multiplier from $\mathcal{K}_{\Phi}$ to $\mathcal{K}_{\Theta}$.

This example is significant since it provides us with an example of two (scalar) inner functions $\Phi$ and $\Theta$ such that $M_{\Phi} \unitary M_{\Theta}$, but so that $\Phi$, the Liv\v{s}ic function for $M_{\Phi}$ does not divide $\Theta$, the Liv\v{s}ic function for $M_{\Theta}$. Indeed, let
$$\Theta(z) = \exp\left(\frac{1 + z}{1 - z}\right)$$ be an atomic inner function. One can show that
\begin{equation}\label{eq:zTT1}
\{\zeta \in \T: \Theta = 1\} = \left\{\frac{2 n \pi - i}{2 n \pi + i}: n \in \Z\right\},
\end{equation}
which is a discrete set of points in $\T$ accumulating only at $\zeta = 1$. Let $F$ be a finite subset of \eqref{eq:zTT1} and construct the inner function $\Phi$ as above. A little thought shows that $\Phi$ is a finite Blaschke product. From the discussion above,
$$G = \frac{1 - \Theta}{1 - \Phi}$$
is an isometric multiplier from $\K_{\Phi}$ into $\K_{\Theta}$.  From this we get that $M_{\Phi} \unitary M_{\Theta}$. But $\Phi$ is a finite Blaschke product and $\Theta$ is a inner function without zeros in $\D$. Thus $\Phi$ does not divide $\Theta$.

If one wanted an example in terms of compressed shifts ($S_u \unitary S_v$ but the inner function $u$ does not divide the inner function $v$) one would need to have $u(0) = v(0) = 0$ which can be accomplished as follows: Let
$$v = z \Theta, \quad u = \frac{\Phi - a}{1 - \overline{a} \Phi},$$
where $a = \Phi(0)$. This makes $u(0) = 0$. If $F$ is the isometric multiplier from $\K_u$ onto $\K_{\Phi}$ (via Crofoot) and $G = (1 - \Theta)(1- \Phi)^{-1}$ then, using the fact that $K_{\Phi} \subset \K_v$, we see that $F G$ is an isometric multiplier from $\K_u$ to $\K_v$ and so $S_{u} \unitary S_v$. However, $u$ is a finite Blaschke product and cannot possibly divide $v$.
\end{Example}

\begin{Example}
Recall the operators $A = S^{*} \oplus S$ on $H^2 \oplus H^2$ and $B = A|_{H^{2} \oplus H^{2}_{0}}$. Notice that the operator $W := I \oplus S$ is an isometry from $\H := H^2 \oplus H^2$ onto $W \H$. Moreover,
\begin{align*}
B  = W A W^{*}
 = S^{*} \oplus S S S^{*}
 = (S^{*} \oplus S) (I \oplus S S^{*})
 = A (I \oplus S S^{*}).
\end{align*}
Notice that $I \oplus S S^{*}$ is the orthogonal projection of $\H$ onto $W \H$. Thus
$$W A W^{*} = A|_{W \H}$$ and so, since $W \H$ is a proper invariant subspace of $A$, it follows that $A$ is unitarily equivalent to a restriction of itself to a proper invariant subspace. One can check
$B \unitary A$ and the associated multiplier from the abstract model space $\HH_{B}$ to $\HH_{A}$ is
$$R(z) = \begin{cases} z &\mbox{if } |z| < 1 ,\\
1 & \mbox{if } |z| > 1. \end{cases}$$
\end{Example}

\begin{Proposition}
    The $\sim _q$ equivalence class, $[A ] _q$, of the partial isometry $A := S ^* \oplus S$ on $H^2 \oplus H^2$ is the unique maximal element of $\VV _1 / \sim _q$ with respect to the partial order $\quasi$. Moreover it
is larger than every other element of $\VV _1 / \sim _q$ with respect to $\quasi$.
\end{Proposition}

\begin{proof}
    This follows straightforwardly from Example \ref{divisor}.  By Example \ref{divisor}, if $\Theta, \Phi$ are contractive analytic functions on $\D$ such that $\Theta $ divides $\Phi$, then
$\ZZZ _\Theta \quasi \ZZZ _\Phi$. The characteristic function of $A$ is $w_A = 0 $, and so any contractive analytic function $b$ divides $w_A$: $w_A = b \cdot 0 = 0 $. It follows that if $V$ is any completely non-unitary
partial isometry with indices $(1,1)$, that its characteristic function $w_V$ divides $w_A$. Hence $V \simeq \ZZZ _{w_V} \quasi \ZZZ _{w_A} \simeq A$, and $V \quasi A$, so that $[V ] _q \quasi [A ] _q $, where $[\cdot ] _q$
denotes $\sim _q$ equivalence class, for any $V \in \VV _1$. It follows that $[A ] _q $ is maximal since if $[ A ] _q \quasi [V] _q $ for some $V \in \VV _1$ then also $[V] _q \quasi [A ] _q$ so that $ [V ] _q = [A ] _q $
since $\quasi$ is a partial order on $\VV _1 / \sim _q$. $[A] _q $ is clearly the unique maximal element since if $[V] _q$ is another maximal element then $[V] _q \quasi [A ] _q$ which implies $[V] _q = [A] _q $ by maximality.
\end{proof}

\begin{Remark}
    Similarly one can show that for any $n \in \N$, the $\sim _q$ equivalence class of $(S^*)^n \oplus S^n$, or equivalently $\left( \bigoplus _{k=1} ^n S^* \right)  \oplus \left( \bigoplus _{k=1} ^n S \right)$ is the unique
maximal element of $\VV _n / \sim _q$ with respect to the partial order $\quasi$.

	By Examples \ref{scalar} and \ref{inner}, if $\Theta, \Phi$ are contractive analytic $M_n (\C )-$valued functions on $\D$ with $\Theta$ inner, and $\Theta$ divides $\Phi$ then
$\ZZZ _\Theta \unitary \ZZZ _\Phi$. It follows as in the proof of the above proposition that the $\sim$ equivalence class $[ n\cdot A ]$ of $n \cdot A  := (S^*) ^n \oplus S^n$ is greater
than that of $V$ with respect to the partial order $\unitary$ on $\VV _n / \sim$ for any $V \in \VV _n$ for which the characteristic function $w_V$ is inner.
\end{Remark}


For $\Theta$ inner, let $\widehat{M} _{\Theta}$ be the multiplication operator on $\K_{\Theta}$ and let $\widehat{\ZZZ} _{\Theta} := \widehat{\ZZZ} _{\widehat{M} _{\Theta}}$ be the abstract model realization of $\widehat{M} _{\Theta}$. Also let $M_{\Theta}$ and $\ZZZ_{\Theta}$ be the partial isometric extensions of $\widehat{M} _{\Theta}$ and $\widehat{\ZZZ} _{\Theta}$. We know that $M_{\Theta}$ and $\ZZZ_{\Theta}$
have the same Liv\v{s}ic characteristic function and thus they are unitarily equivalent.

Furthermore, by Section \ref{POmult}, for two inner functions $\Theta$ and $\Phi$ we have that $M_{\Theta} \unitary M_{\Phi}$ if and only if there is an isometric multiplier from $\K_{\Theta}$ to $\K_{\Phi}$. Thus we see that $\ZZZ_{\Theta} \unitary \ZZZ_{\Phi}$ (which is equivalent to the fact that there is an isometric multiplier from $\HH_{M_{\Theta}}$ to $\HH_{M_{\Phi}}$) if and only there is an isometric multiplier from $\K_{\Theta}$ to $\K_{\Phi}$. This relates the isometric multiplier problem in the abstract setting to the one explored by Crofoot \cite{MR1313454}.

\begin{Example}
    Consider the partial isometries $M_B $, which act as multiplication by $z$ on their initial spaces in a model space $\K _B$ where
$B$ is a finite Blaschke product.  This example will show three things.
First we will show that $M _1  := M_{B_1} \quasi  M_2 := M_{B_2}  $ if and only if the degree of $B_2$ (number of zeroes) is greater than that of $B_1$, demonstrating that the partial order $\quasi$ is somewhat trivial when restricted to such partial isometries. Next we provide an example of $M_1  \unitary M_2  $ for finite Blaschke products $B_1, B_2$ even
though $B_1$ does not divide $B_2$. Finally we will show that there exist $B_1 , B_2$ so that the degree of $B_1$ is less than that of $B_2$ but $M_1 $ is not less than $M_2 $ with
respect to $\unitary$.  This will show that the two partial orders $\unitary$ and $\quasi$ are different.

Let $B_1, B_2$ be finite Blaschke products of degree $n \leqslant m$ and zero sets $\{ z_1, ... , z_n \}$ and $\{ w_1 , ... , w_m \}$ respectively. Then
$$ \K _{B_1} = \left\{ \left. \frac{p(z)}{(1 - \overline{z_1} z )\cdots(1 - \overline{z_n} z) } \ \right|  \ p \in \C [z]; \ \mbox{deg($p$)}\leqslant n-1 \right\}, $$ and similarly for $\K _{B_2}$. Since $n \leqslant m$,
the function $$ R(z) := \frac{(1 - \overline{z_1} z) \cdots (1-\overline{z_n} z) }{(1- \overline{w_1} z ) \cdots (1- \overline{w_m} z ) }, $$ is analytic, bounded on $\D$, and is a multiplier from $\K _{B_1}$ into $\K _{B_2}$. By Theorem \ref{multiplier}, and the discussion above, we have that $M_1  \quasi M_2 $.

Let
$$B_1(z) = z^2, \quad B_2(z) = z \frac{z - a}{1 - \overline{a} z}, \quad a \not = 0.$$ Note that
$$\K_{B_1} = \{d_0 + d_1 z: d_0, d_1 \in \mathbb{C}\},$$
$$\K_{B_2} = \{\frac{c_0 + c_1 z}{1 - \overline{a} z}: c_0, c_1 \in \mathbb{C}\}.$$
Thus, as just seen above, if
$$\phi = \frac{1}{1 - \overline{a} z},$$
we clearly have $\phi \K_{B_1} \subset \K_{B_2}$. In fact $\phi \K_{B_1} = \K_{B_2}$ as a bonus. Hence there is a multiplier from $\K_{B_1}$ to $\K_{B_2}$. However, there is no isometric multiplier from $\K_{B_1}$ to $\K_{B_2}$ and thus $M_{1} \quasi M_2$ but $M_1 \not \unitary M_2$. To see this, observe that since  $\mathbb{C} \subset \K_{B_1}$ we see that any multiplier $\phi$ from $\K_{B_1}$ to $\K_{B_2}$ satisfies $\phi \in \K_{B_2}$. Thus
$$\phi = \frac{c_0 + c_1 z}{1 - \overline{a} z}.$$
Notice that $c_1 = 0$ since otherwise $z \phi \not \in \K_{B_2}$. Thus $\phi$ takes the form
$$\phi = \frac{c_0}{1 - \overline{a} z}.$$
If $\phi$ is an isometric multiplier then $\phi$ must satisfy the identities
$$\|\phi 1\| = \|1\| = 1, \quad \langle \phi z, \phi 1\rangle = \langle z, 1 \rangle = 0.$$
The first identity says that
$$1 = \int |\phi|^2 dm = |c_0|^2 \int \frac{1}{|1 - \overline{a} z|^2} dm = |c_0|^2 \frac{1}{1 - |a|^2}$$
and so
$$\phi = \zeta \frac{\sqrt{1 - |a|^2}}{1 - \overline{a} z}$$
for some unimodular constant $\zeta$.
The second identity says that
$$0 = \int |\phi|^2 z dm = \int \frac{1 - |a|^2}{|1 - \overline{a} z|^2} z dm.$$
Notice how the above integral is the Poisson integral of the function $z$ (which is certainly harmonic on the disk) and so it evaluates to $a$.
Thus $a = 0$ which yields a contradiction. Thus there there is a multiplier from $\K_{B_1}$ to $\K_{B_2}$ but no isometric multiplier.
\end{Example}

\section{Equivalence classes}

We have defined two equivalence classes $\sim$ and $\sim_q$ on $\VV_{n}$ by declaring $A \sim B$ if $A \unitary B$ and $B \unitary A$ (respectively $A \sim_q B$ if $A \quasi B$ and $B \quasi A$). Can we precisely identify these equivalence classes? In some cases we can.

\begin{Theorem}
Suppose $A, B \in \VV_1$ with inner Livsic functions. Then $A \sim B$ if and only if $A$ is unitarily equivalent to $B$. \label{eclass}
\end{Theorem}

\begin{proof}
If $A \unitary B$, then there is an isometric multiplier $m_{A}$ from $\K_{\Theta_A}$ to $\K_{\Theta_B}$. Likewise if $B \unitary A$, then there is an isometric multiplier $m_{B}$ from $\K_{\Theta_B}$ to $\K_{\Theta_A}$. The product $m := m_{A} m_{B}$ is a multiplier from $\K_{\Theta_A}$ to itself. By a theorem of Crofoot, $m$ must be a constant function with unimodular constant. Furthermore, $m \K_{\Theta_A} = \K_{\Theta_A}$.

We now claim that $m_A \K_{\Theta_{A}} = \K_{\Theta_B}$. Let $g \in \K_{\Theta_{B}}$. Then $m_{A} g = f \in \K_{\Theta_{B}}$ and so $m g = m_B m_{A} g = m_{B} f \in \K_{\Theta_{A}}$. But then $g = m^{-1} m_{B} f \in \K_{\Theta_{A}}$ and $m_{A} g = f$.

But then the isometric operators $\widehat{\ZZZ} _{A}$ and $\widehat{\ZZZ} _{B}$ are unitarily equivalent via
$$X: \K_{\Theta_{A}} \to \K_{\Theta_{B}}, \quad X f = m_A f.$$
 Since $A$ and $\ZZZ_{A}$ are unitarily equivalent and since $B$ and $\ZZZ_{B}$ are unitarily equivalent, we see that $A$ and $B$ are unitarily equivalent.

 The converse is obvious.
\end{proof}

Is turns out that this result can be extended beyond $n = 1$ by applying the theory of \cite{Martin-ext}, but the proof is much more involved and we will not include it here.

\begin{Theorem}
Suppose $A, B \in \VV_n$ with inner Livsic functions. Then $A \sim B$ if and only if $A$ is unitarily equivalent to $B$. \label{eclass2}
\end{Theorem}

\begin{Theorem}
Suppose $A, B \in \VV_1$ with inner Livsic functions. Then $A \sim_q B$ if and only if $A|_{\ker{A}^{\perp}}$ is similar to $B|_{\ker{B}^{\perp}}$.
\end{Theorem}

\begin{proof}
Essentially the same argument as above shows that $\widehat{\ZZZ} _{A}$ is similar to $\widehat{\ZZZ} _{B}$ via the invertible multiplier $m_A$, i.e., $Y: \K_{\Theta_{A}} \to \K_{\Theta_{B}}$, $Y f = m_A f$. Moreover, $Y \dom{\widehat{\ZZZ} _{A}} = \dom{\widehat{\ZZZ} _{B}}$.

If $U_{A}: \H_{A} \to \K_{\Theta_{A}}$ (here $\H_{A}$ is
the Hilbert space on which $A$ acts) is the unitary operator which induces the unitary equivalence of $A$ and $\ZZZ_{A}$
and $U_{B}: \H_{B} \to \K_{\Theta_B}$
is the unitary inducing the unitary equivalence of $B$ and $\ZZZ_{B}$, one notes that by the way in which these operators were constructed, we have
$$U_{A} \ker{A}^{\perp} = \dom{\widehat{\ZZZ} _{A}}, \quad U_{B} \ker{B}^{\perp} = \dom{\widehat{\ZZZ} _{B}}.$$

One can verify that the operator $L = U_{B}^{*} Y U_{A}: \H_{A} \to \H_{B}$ satisfies
$$L \ker{A}^{\perp} = \ker{B}^{\perp}, \qquad L A|_{\ker{A}^{\perp}} = B L|_{\ker{A}^{\perp}}.$$
This shows that $A|_{\ker{A}^{\perp}}$ is similar to $B|_{\ker{B}^{\perp}}$.

As in the previous proof, the converse is obvious.
\end{proof}

\begin{Example}
Let $\{\mathbf{e}_1, \ldots, \mathbf{e}_n\}$ be the standard orthonormal basis for $\C^n$ and let $\{\mathbf{u}_1, \ldots, \mathbf{u}_n\}$ be any orthonormal basis for $\C^n$. By Proposition \ref{canonicalformPI} the matrices
$$V_{1} = [\mathbf{e}_2|\mathbf{e}_3|\cdots|\mathbf{e}_n|\mathbf{0}], \quad V_{2} = [\mathbf{u}_2|\mathbf{u}_3|\cdots|\mathbf{u}_n|\mathbf{0}]$$
define partial isometries on $\C^n$. Note that $V_1$ is the matrix representation of the compressed shift $S_{\Theta}$ on $\K_{\Theta}$, where $\Theta = z^n$.

From Example \ref{poiipoipoi}, we see that the Liv\v{s}ic characteristic function for $V_1$ is $\Theta$ while the Liv\v{s}ic characteristic function for $V_2$ is the finite Blaschke product $\Psi$ whose zeros are $0$ along with the non-zero eigenvalues of $V_2$ (Example \ref{sdkfhjadskfjh}). So unless $\Theta = \xi \Psi$, for some $\xi \in \T$, $V_1$ is not unitarily equivalent to $V_2$ (Theorem \ref{LivsicThm}). However, we can see that $V_1 \sim_q V_2$ in the following way.

Observe from \eqref{weoiruwoeiru} that
$$\ker{V_1}^{\perp} = \bigvee\{\mathbf{e}_1, \ldots, \mathbf{e}_{n - 1}\}, \quad \ran{V_1}  = \bigvee \{\mathbf{e}_2, \ldots, \mathbf{e}_n\}.$$
Furthermore,
$V_1 \mathbf{e}_j = \mathbf{e}_{j + 1}, 1 \leqslant j \leqslant n - 1$.
This means that if $\mathscr{B}_1$ is the ordered basis $\{\mathbf{e}_1, \ldots, \mathbf{e}_{n - 1}\}$ for $\ker{V_1}^{\perp}$
and $\mathscr{B}_2$ is the ordered basis $\{\mathbf{e}_2, \ldots, \mathbf{e}_n\}$ for $\ran{V_1}$, then
the matrix representation of $V_{1}|_{\ker{V_1}^{\perp}}$ with respect to the pair $(\mathscr{B}_1, \mathscr{B}_2)$ is
$$[V_{1}|_{\ker{V_1}^{\perp}}]_{(\mathscr{B}_1, \mathscr{B}_2)} = I_{n - 1}.$$
In a similar way,
$$\ker{V_2}^{\perp} = \bigvee\{\mathbf{e}_1, \ldots, \mathbf{e}_{n - 1}\}, \quad \ran{V_2} = \bigvee\{\mathbf{u}_1, \ldots, \mathbf{u}_{n - 1}\}.$$
Moreover, $V_{2} \mathbf{e}_j = \mathbf{u}_j, 1 \leqslant j \leqslant n - 1$.
This means that if $\mathscr{C}_1$ is the ordered basis $\{\mathbf{e}_1, \ldots, \mathbf{e}_{n - 1}\}$ for $\ker{V_2}^{\perp}$
and $\mathscr{C}_2$ is the ordered basis $\{\mathbf{u}_1, \ldots, \mathbf{u}_{n - 1}\}$ for $\ran{V_2}$, then
the matrix representation of $V_{2}|_{\ker{V_2}^{\perp}}$ with respect to the pair $(\mathscr{C}_1, \mathscr{C}_2)$ is
$$[V_{2}|_{\ker{V_2}^{\perp}}]_{(\mathscr{C}_1, \mathscr{C}_2)} = I_{n - 1}.$$
Since we get the $(n - 1) \times (n - 1)$ identity matrix in both cases, we see, from basic linear algebra, that
$V_{1}|_{\ker{V_1}^{\perp}}$ is indeed similar to $V_{2}|_{\ker{V_2}^{\perp}}$.

\end{Example}

\bibliography{PartialOrder}

\end{document}